\documentclass{article}

\usepackage{arxiv}          
\usepackage{texcommands}    
\usepackage[utf8]{inputenc} 
\usepackage[T1]{fontenc}    
\usepackage{hyperref}       
\usepackage{url}            
\usepackage{booktabs}       
\usepackage{amsfonts}  
\usepackage{mathrsfs}       %
\usepackage{mathtools}
\usepackage{amsthm}         
\usepackage{amssymb}         
\usepackage{graphicx}
\usepackage[dvipsnames]{xcolor}         
\usepackage{upgreek} 
\usepackage{bbm}
\usepackage{verbatim} 
\usepackage{soul} 
\usepackage[linesnumbered,ruled,vlined]{algorithm2e}
\usepackage[toc,title]{appendix} 
\usepackage{derivative} 
\usepackage{subfigure}

\usepackage{xcolor} 

\theoremstyle{plain}
\newtheorem{theorem}{Theorem}[section]
\newtheorem{lemma}[theorem]{Lemma}
\newtheorem{proposition}[theorem]{Proposition}

\newtheorem{assumption}[theorem]{Assumption}

\theoremstyle{definition}

\newtheorem{example}[theorem]{Example}
\newtheorem{remark}[theorem]{Remark}

\DeclareMathOperator{\dom}{dom}

\DeclareMathOperator{\tr}{tr}

\DeclareMathOperator{\solve}{solve}

\title{Simulation Of Infinite-Dimensional Diffusion Bridges}

\author{
  Thorben Pieper-Sethmacher \\
  Delft Institute for Applied Mathematics \\
  Delft University of Technology \\
    Mekelweg 4, 2628 CD Delft \\
    The Netherlands \\
  \texttt{t.pieper@tudelft.nl} \\
   \And
  Frank van der Meulen \\
  Department of Mathematics \\
  Vrije Universiteit Amsterdam \\
    De Boelelaan 1111, 1081 HV Amsterdam\\
    The Netherlands \\
  \texttt{f.h.van.der.meulen@vu.nl} \\
     \And
  Aad van der Vaart \\
  Delft Institute for Applied Mathematics \\
  Delft University of Technology \\
    Mekelweg 4, 2628 CD Delft \\
    The Netherlands \\
  \texttt{a.w.vandervaart@tudelft.nl} \\
}

\begin{document}
\maketitle
\begin{abstract}
Let $X$ be the mild solution to a semilinear stochastic partial differential equation. 
In this article, we develop methodology to sample from the infinite-dimensional diffusion bridge that arises from conditioning $X$ on a linear transformation $L X_T$ of the final state $X_T$ at some time $T> 0$. This solves a problem that has so far not been attended to in the literature. Our main contribution is the derivation of a path measure that is absolutely continuous with respect to the path measure of the infinite-dimensional diffusion bridge. 
This lifts previously known results for stochastic ordinary differential equations to the setting of infinite-dimensional diffusions and stochastic partial differential equations.
We demonstrate our findings through numerical experiments on stochastic reaction-diffusion equations.
\end{abstract}

\keywords{Doob's h-transform; Exponential change of measure; Guided process; Infinite-dimensional diffusion bridge; Monte Carlo methods; Semilinear SPDE; SPDE bridge; Stochastic Allen-Cahn equation, Stochastic Amari equation}

\section{Introduction}
\label{sec: 1_intro}

Consider a semilinear \textit{stochastic partial differential equation (SPDE)} of the form
\begin{align}
\label{eq: semilin_spde}
	\begin{split}
		\begin{cases}
			\df X_t &= \left[  A X_t + F(t,X_t) \right] \df t + Q^{\fsqrt} \df W_t, \quad t \geq 0, \\
			X_0 &= x_0.
		\end{cases}
	\end{split}
\end{align}
The operator $A$ denotes the generator of a strongly continuous semigroup $(S_t)_{t \geq 0}$ on a Hilbert space $H$, whereas $F$ denotes a nonlinear operator and $Q$ is a symmetric, positive operator on $H$. The process $W$ is a cylindrical Wiener process on $H$, defined on a stochastic basis $(\Om, \calF, (\calF_t)_{t \geq 0},\P)$. We assume that the operators $A$, $F$ and $Q$ satisfy suitable conditions such that Equation \eqref{eq: semilin_spde} admits a unique mild solution $X = (X_t)_{t \geq 0}$ for any $x_0 \in H.$

For some $k \geq 1$, let $L: H \to \R^k$ denote a bounded, linear operator, termed the observation operator. Suppose we observe the linear transformation $LX_T$ of $X_T$ at some time $T >0$. 
This includes the highly relevant cases that we observe an orthogonal projection of $X_T$ onto a finite-dimensional subspace or that, in the case of $H = L^2(D)$ for some $D \subset \R^d$, we observe weighted integrals of $X_T$.

We are interested in sampling the process $X$ on the interval $[0,T]$, conditioned on the event $\{L X_T = y\}$ for some fixed $y \in \R^k$. This problems arises, for example, in parameter estimation for partially observed SPDEs. 
Following the literature for diffusion processes with finite-dimensional state spaces, we refer to the conditioned process as an \textit{infinite-dimensional diffusion bridge} and denote it by $\Xstr$.
To the best of our knowledge, the problem of drawing samples of $\Xstr$ has not been considered before in the generality of our setup.

This stands in contrast to the corresponding problem for a stochastic \emph{ordinary} differential equation, which has attracted considerable attention in the past two decades (see for instance \cite{Delyon2006simulation}, \cite{Beskos08MCMC}, \cite{Papaspiliopoulos2012importance},  \cite{vdMeulenSchauer2017Guided}, \cite{vdMeulenBierkensSchauer2020Simulation}, \cite{Heng2021simulating}).
By Doob's $h$-transform (\cite{Rogers2000diffusions}, Section IV.39), it can be shown that a \textit{diffusion bridge} $\xstr$ associated with the strong solution $x$ to the SODE
\begin{align*}
        \df x_t =  f(x_t) \df t + \sigma \df w_t
\end{align*}
is characterized by yet another SODE given by 
\begin{align*}
    \df \xstr_t = [f(\xstr_t) + \sigma \sigma^{*} \Df_x \log h(t,\xstr_t)] \df t +  \sigma \df w_t. 
\end{align*}
Here, $h$ is a mapping that depends on both the transition densities $p$ and the matrix representation of the map $L$. 
In particular, the conditioned process $\xstr$ solves an SODE itself, where the additional term $ \sigma \sigma^{*} \Df_x \log h(t,\xstr_t)$ is added to the drift coefficient to steer the process towards the point of conditioning.

Unfortunately, the transition density $p$, and consequently the function $h$, is not known in closed form except for a few special cases. 
This has motivated several works in which the map $h$ is replaced by a substitute $g$ in the steering term (\cite{Clark1990simulation}, \cite{Delyon2006simulation}, \cite{vdMeulenSchauer2017Guided}, \cite{Mider2021continuous}). If $g$ is chosen carefully, then indeed the laws of the conditioned process and the process that substitutes $g$ for $h$ are absolutely continuous on path space. In particular, \emph{weighted samples} of the conditioned process can be obtained. This has important applications for example in Bayesian parameter estimation for discretely observed diffusion processes (see for instance \cite{Roberts2001inference} and \cite{Papaspiliopoulos2013data}).

The literature on the corresponding problem for infinite-dimensional diffusions is scarce. Notable exceptions are \cite{DinunnoPetersson23} and \cite{Yang2024Simulating}. The former considers the situation where $F \equiv 0$, in which case the diffusion bridge is given by a conditioned Gaussian process on $H$. Moreover, convergence rates are derived for the spatial discretisation of $\Xstr$. 
In \cite{Yang2024Simulating}, the process $X$ is conditioned on a set of positive measure and projected onto a finite-dimensional subspace of $H$ after which the problem is considered in finite dimensions. Additionally, Equation \eqref{eq: semilin_spde} is assumed to admit a strong solution, an assumption that is generally not satisfied when working with SPDEs.

\subsection{Contribution}

In this paper, we lift earlier results obtained for SODEs in \cite{vdMeulenSchauer2017Guided} and \cite{vdMeulenBierkensSchauer2020Simulation} to SPDEs, resulting in computational methods for sampling infinite-dimensional diffusion bridges. 
This extension is nontrivial as we need to deal with complications arising from working in infinite-dimensional state spaces. 
We strongly depend on our earlier work in \cite{Piepersethmacher24Classexponentialchangesmeasure} on exponential change of measure for SPDEs. Here,  it was shown that the infinite-dimensional diffusion bridge $\Xstr$ can be obtained by an \textit{exponential change of measure} $\Pstr$ on the underlying filtered probability space. 
Under additional assumptions, it was shown that $\Xstr$ is equivalent in law to the mild solution of the \textit{bridge SPDE}
\begin{align*}
        \df \Xstr(t) = \left[ A \Xstr(t) + F(t,\Xstr(t)) + Q \Df_x \log h(t,\Xstr(t))\right] \df t + Q^{\fsqrt} \df W(t), \quad t<T,
\end{align*}
where $h(t,z) := \rho_X(t,z;T,y)$ denotes the density of $L X_T \mid X_t = z$, evaluated at $y$. 
As in the finite-dimensional case, $\rho_X$ is in general intractable, rendering a direct simulation of $\Xstr$ infeasible. As a further complication, the formulation of the bridge SPDE relies on the existence of the gradient $\Df_x \log h$ in the first place. However, to the best of our knowledge, no such regularity result for transition densities of SPDEs of the form \eqref{eq: semilin_spde} has been obtained in the literature so far. 
On the other hand, the \textit{Ornstein-Uhlenbeck} process $Z$, which satisfies the linearized version of \eqref{eq: semilin_spde} with $F \equiv 0$, is a Gaussian process for which the  the density of $LZ_T \mid Z_t=z$ is tractable. Moreover, for this density, which we denote by $\rho_Z$, the required regularity assumptions can be validated. 
This motivates substituting $\rho_X$ by $\rho_Z$, by which we obtain the \textit{guided process} $\Xcrc$ as the mild solution to the SPDE
\begin{align*}
     \df \Xcrc(t) = \left[ A \Xcrc(t) + F(t,\Xcrc(t)) + Q \Df_x \log g(t,\Xcrc(t))\right] \df t + Q^{\fsqrt} \df W(t),
\end{align*}
where $g(t,x)=\rho(t,x; T,y)$. 
In \cite{Piepersethmacher24Classexponentialchangesmeasure}, it was shown that $\Xcrc$ equals in law the process $X$ under yet another exponential change of measure $\Pcrc$. 
Moreover, it was shown that the laws $\calLstr$ and $\calLcrc$ of $\Xstr$ and $\Xcrc$ are absolutely continuous on the path space $C([0,t];H)$ for any $t < T.$
Our main result in this work, presented in Theorem \ref{thm: absolute_continuity}, shows that the absolute continuity persists in the limit $t \uparrow T$. 
This facilitates $\calLcrc$ as a suitable proposal distribution for $\calLstr$, hence enabling a wide range of computational methods to sample from infinite-dimensional diffusion bridges. These include, for example, Markov chain Monte Carlo algorithms, importance sampling or variational methods. Algorithm \ref{alg: MH_sampler1} shows how  Theorem \ref{thm: absolute_continuity} can be used to define a Metropolis-Hastings algorithm for sampling diffusion bridges for SPDEs. We show numerical performance in stochastic reaction diffusion equations, including the Allen-Cahn equation.  
A crucial step towards the main result concerns the convergence of the guided process $\Xcrc(t)$ towards the conditioning point $y$ at an appropriate rate (Cf.\  Theorem \ref{thm: Xcrc_limit_behavior}). While the main steps in the  proof are similar to the proof of absolute continuity in \cite{vdMeulenBierkensSchauer2020Simulation}, complications arise from dealing with the infinite-dimensional setting. As it is unclear under which conditions transition densities exist, contrary to \cite{vdMeulenBierkensSchauer2020Simulation},  the proof of Theorem \ref{thm: absolute_continuity} avoids their existence. As we restrict to additive noise, we also obtain much ``cleaner'' conditions for absolute continuity to hold. Our examples show that Assumption \ref{eq: ass_Xcrc_limit_behavior}  can easily be verified in the diagonalisable case (Cf.\ Section \ref{sec: 5_examples}). The other assumption, Assumption \ref{ass: absolute_continuity},  is shown to hold under the easily verified conditions of Lemma \ref{lem: bounds_gh}.

\subsection{Outline}
The necessary preliminaries are outlined in Section \ref{sec: 2_prelims}. This contains a brief description on the derivation of the infinite-dimensional diffusion bridge and the guided process. 
In Section \ref{sec: 3_main_results} we state the main results of the paper and in Section \ref{sec: 4_MH_algorithm} we present a Metropolis-Hastings algorithm for sampling infinite-dimensional diffusion bridges.
We showcase the performance of this algorithm in numerical experiments in Section \ref{sec: 5_examples}. The proofs of our main results, Theorem \ref{thm: Xcrc_limit_behavior} and Theorem \ref{thm: absolute_continuity}, are given in Section \ref{sec: 6_proof_limit_behavior} and Section \ref{sec: 7_proof_abs_cont} respectively.

\subsection{Frequently used notation}
For the readers convenience we summarize here the general notation as well as processes and operators that are frequently used throughout this work. 

\paragraph{General}
We denote by $H$ a Hilbert space equipped with inner product $\langle x, y \rangle$ and norm $|x| = \sqrt{\langle x, x \rangle}$. 
For any linear, bounded operator $B: H \to H$ we let $\|B\|$ be the operator norm of $B$. If $B$ is a trace-class operator, we denote by $\tr[B]$ its trace and if $B$ is a Hilbert-Schmidt operator we write $\|B\|_{\text{HS}}$ for its Hilbert-Schmidt norm. Moreover, $B^*$ denotes the adjoint of $B$, and if $B$ is positive, $B^{\fsqrt}$ denotes the square root operator of $B$. 

The space $C([0,T];H)$ denotes the space of continuous, $H$-valued function, endowed with the supremum norm $\|\varphi\| = \sup_{t \in [0,T]} |\varphi(t)|$. Without further mention, we assume that all normed spaces are equipped with their corresponding Borel $\sigma$-Algebra.

For any arbitrary set $\mathcal{X}$ and two functions $f,g: \mathcal{X} \to \R$, we write $f(x) \lesssim g(x)$ if there exists some constant $C > 0$, independent of $x$, such that $f(x) \leq C g(x)$ for all $x$. If $\mathcal{X}$ equals $\R$ or $\N$ and $f$ and $g$ are asymptotically equivalent, we write $f \sim g$. 
The tuple $(\Om, \calF, (\calF_t)_{t \geq 0}, \P)$ denotes a stochastic basis and we let $\P_t$ be the restriction of $\P$ onto $\calF_t$. 

\paragraph{Operators and processes}
For any $k \geq 1$, the operator $L: H \to \R^k$ is a bounded, linear operator that we refer to as the observation operator.
The operator $(A,\dom(A))$ is the generator of a strongly continuous semigroup $(S_t)_{t \geq 0}$ on $H$. Moreover, $Q$ is a positive symmetric operator and $F$ a continuous nonlinearity, both acting on $H$ as well. The process $X$ denotes the mild solution to the SPDE \eqref{eq: semilin_spde}, whereas $Z$ is the Ornstein-Uhlenbeck process, i.e. the mild solution to \eqref{eq: semilin_spde} for the case that $F \equiv 0$.
By $Q_t$ we denote the covariance operator of $Z_t$, whereas $R_t = L Q_t L^*$ and $L_t = L S_t$. 
The process $\Xstr$ is the conditioned process $X$ given $L X_T = y$ and $\Xcrc$ the guided process, respectively derived from $X$ under the changes of measure $\Pstr$ and $\Pcrc$ on $\calF_T$. 

For the process $(X_t)_{t \in [0,T]}$ we denote by $\calL(X)$ the law of $X$ on the path space $C([0,T];H)$ and by $\calL_t(X)$ the restriction of its law on $C([0,t];H)$, whereas $\calL(X_t)$ is the law of $X_t$ on $H$.
Finally, for the measures $\calL(\Xstr)$ and $\calL(\Xcrc)$ we shorten notation and simply write $\calLstr$ and $\calLcrc$ instead.

\section{Preliminaries}
\label{sec: 2_prelims}

\subsection{Mild solutions}
\label{subsec: 2_mild_solutions}

The following is a standing assumption on the components involved in Equation \eqref{eq: semilin_spde}.
\begin{assumption} ~
\label{ass: basic_assumptions_SPDE}
    \begin{itemize}
        \item[(i)] $A$ is the generator of a strongly continuous semigroup $(S_t)_{ t \geq 0}$ on $H$. In particular there exists a $C_S > 0, \om_S \in \R$ such that $\|S_t\| \leq C_S \exp(\om_S t)$ for all $t \geq 0.$
        \item[(ii)] $W$ is a cylindrical Wiener process on $H$. 
        \item[(iii)] $Q$ is a symmetric, positive and bounded operator on $H$. Moreover, the operators $(Q_t)_{t \geq 0}$ given by
        \begin{align}
        \label{eq: def_Q(t)}
            Q_t = \int_0^t S_u Q S^*_u \, \df u 
        \end{align}
        are well-defined and trace-class.
        \item[(iv)] $F$ is such that there exists a constant $C_{F} > 0$ with 
        \begin{align*}
                    | F(t,x) - F(t,y) | \leq C_{F} | x- y | \text{ and }         | F(t,x) | \leq C_{F} 
        \end{align*}
        for all $t \in [0,T]$ and $ x,y \in H.$
    \end{itemize}
\end{assumption}

Under Assumption \ref{ass: basic_assumptions_SPDE}, Equation \eqref{eq: semilin_spde} admits, for any initial value $x_0 \in H$, a unique mild solution $X$ that satisfies
\begin{align}
\label{eq: mild_sol}
    X_t = S_t x_0 + \int_0^t S_{t-s}F(s,X_s) \, \df s + \int_0^t S_{t-s} Q^{\fsqrt} \, \df W_s, \quad t \geq 0.
\end{align}
The process $X$ is Markovian and has an almost surely continuous modification. 
In the special case that $F = 0$ we denote the time homogeneous mild solution to equation \eqref{eq: semilin_spde} by $Z$ and refer to it as the \textit{Ornstein-Uhlenbeck (OU) process}. The Ornstein-Uhlenbeck process is a Gaussian process and, in particular, for any $s < t$ and $x \in H$, the distribution of $Z_t \mid Z_s = x$ is Gaussian with mean $S_{t-s} x$ and covariance operator $Q_{t-s}$.
Moreover, this implies that the distribution of  $L Z_t \mid Z_s =x$ is  Gaussian in $\R^k$ with mean $L_{t-s}x$ and covariance matrix $R_{t-s}$, with operators $L_t$ and $R_t$ defined by 
\begin{align}
\begin{split}
\label{eq: L_t,R_t} 
L_{t} &= L S_{t},\\
R_{t} &= L Q_{t} L^*, \quad t \in [0,T].
\end{split}
\end{align}

\subsection{Derivation of the diffusion bridge and guided process}
\label{subsec: 2_expchangemeasure}
For the reader's convenience, we give an outline of the derivation of the infinite-dimensional diffusion bridge $\Xstr$ and the guided process $\Xcrc$ based on an exponential change of measure as introduced in \cite{Piepersethmacher24Classexponentialchangesmeasure}. The details of this subsection's proofs can be found in Appendix \ref{app: A}.

For this, let $(T_t)_{t \geq 0}$ denote the transition semigroup 
\begin{align*}
    (T_t \varphi)(s,x) &= \E[\varphi(s+t,X_{t+s}) \mid X_s = x] 
\end{align*}
defined on the Banach space $C_m(\R_+ \times H)$ of continuous functions $\varphi: \R_+ \times H \rightarrow \R$ such that $\|\varphi\|_m = \sup_{t,x} (1 + \|x\|^m)^{-1} | \varphi(t,x)| < \infty.$ 
A sequence $(\varphi_n)_n \subset C_m(\R_+ \times H)$ is said to converge to $\varphi \in C_m(\R_+ \times H)$ in the topology of \textit{bounded-pointwise} convergence if $\lim_n \varphi_n(t,x) = \varphi(t,x)$ for any $t \geq 0, x \in H$ and $\sup_n \|\varphi_n\|_m < \infty.$ We then write $\pilim_n \varphi_n = \varphi.$

In the topology of bounded-pointwise convergence, define the infinitesimal generator $(K, \dom_m(K))$ of $(T_t)_t$ via
\begin{align}
\label{eq: def_inf_gen}
\begin{split}
\begin{cases}
    \dom_m(K) &= \left\{ \varphi \in C_m(\R_+ \times H): \exists \psi \in C_m(\R_+ \times H) \text{ s.t. } \pilim_{t \downarrow 0} \dfrac{T_t \varphi - \varphi}{t} = \psi \right\} \\
    (K \varphi)(s,x) &= \lim_{t \downarrow 0} \dfrac{(T_t \varphi)(s,x) - \varphi(s,x)}{t}, \quad \varphi \in \dom_m(K), (s,x) \in \R_+ \times H.
\end{cases}
\end{split}
\end{align}

For any positive function $h \in \dom_m(K)$, define the process 
\begin{align}
\label{eq: E^h_t}
        E^h_t = \dfrac{h(t,X_t)}{h(0,x_0)} \exp\left( - \int_0^t \dfrac{K h}{h}(s,X_s) \, \df s \right).
\end{align}
It can be shown that $E^h$ is a continuous local $\P$-martingale whenever it exists, see for example Lemma 3.1 in \cite{PalmowskiRoski2002technique}.
If $E^h$ is a true $\P$-martingale, it defines an \textit{exponential change of measure} $\P^h$ on $\calF_T$. Moreover, under additional assumptions on $h$, the change of measure $\P^h$ is of Girsanov-type, as was shown in the following theorem. 

\begin{theorem}[Theorem 3.5, \cite{Piepersethmacher24Classexponentialchangesmeasure}]
\label{thm: htransform}
Let $h: [0,T) \times H \longrightarrow \R_{>0}$ satisfy the following for any $S < T$:
\begin{itemize}
    \item[(i)] $h \in C_m([0,S] \times H)$ such that $K h(t,x)$ exists for any $[0,S] \times H$ and $h^{-1} K h \in C_m([0,S] \times H)$.
    \item[(ii)] $h$ is such that the process $(E^h_t)_{t \in [0,S]}$ is a $\P$-martingale. 
\end{itemize}
Then $h$ defines a unique change of measure $\P^h$ on $(\Om, \calF_T)$ via 
\begin{align*}
    \df \P^h_t = \E^h_t \df \P_t, \quad t < T.
\end{align*}
Additionally, let $h$ satisfy the following:
\begin{itemize}
    \item[(iii)] $h$ is Fréchet differentiable in $x$ such that $\Df_x h \in C_m([0,S] \times H;H)$ for any $S < T$.
\end{itemize}
Then $X$ under $\P^h$ is a mild solution to the SPDE
\begin{align*}
    \df X_t = \left[ A X_t + F(t,X_t) + Q \Df_x \log h(t,X_t) \right] \df t + Q^{\fsqrt} \df W^h_t, \quad t \in [0,T),
\end{align*}
where $W^h$ is a $\P^h$-cylindrical Wiener process.
\end{theorem}

The diffusion bridge $\Xstr$ and corresponding guided process $\Xcrc$ are defined through an application of Theorem \ref{thm: htransform} for different choices of the function $h$ as follows. 
Let $y \in \R^k$ and $T>0$ be fixed. For any $x \in H$ and $t \in [0,T)$, denote by $\rho_Z(t,x;T,y)$ the density of $L Z_T \mid Z_t = x$ with respect to the Lebesgue measure on $\R^k$, evaluated at $y$. As stated in Section \ref{subsec: 2_mild_solutions}, $\rho_Z(t,x;T,y)$ is Gaussian in $y$ with mean $L_{T-t}x$ and covariance matrix $R_{T-t}$ defined in \eqref{eq: L_t,R_t}.

Moreover, let $\rho_X(t,x;T,y)$ denote the density of $L X_T \mid X_t = x$, evaluated at $y$. The existence of $\rho_X(t,x;T,y)$ is a consequence of the Girsanov theorem, from which the absolute continuity of $\calL(X_T)$ with respect to $\calL(Z_T)$, and hence of $\calL(LX_T)$ with respect to $\calL(LZ_T)$ follows.

On $[0,T) \times H$, define the functions $h$ and $g$ by
\begin{align}
\label{eq: def_h_g}
\begin{split}
  h(t,x) &= \rho_X(t,x;T,y), \\
  g(t,x) &= \rho_Z(t,x;T,y).
\end{split}
\end{align}

\begin{remark}
The function $h(t,x) = \rho_X(t,x;T,y)$ is only defined $\calL(LX_T)$-almost surely in $y$, i.e. there exists a Borel set $\calA_0$ such that $\P( L X_T \in \calA_0) = 1$ and $(t,x,y) \mapsto \rho_X(t,x;T,y)$ is uniquely well-defined for all $t \in [0,T), x \in H$ and $y \in \calA_0$. The existence of such a set, independently of $(t,x)$, follows from the Girsanov theorem and the continuity of $(t,x) \mapsto LZ_T \mid Z_t = x$.
Throughout the remainder of this work, we assume without further notice that $y \in \calA_0$.
\end{remark}

The following two propositions verify that $h$ and $g$ satisfy the assumptions of Theorem \ref{thm: htransform}, and consequently, prove the existence of $\Xstr$ and $\Xcrc$.

\begin{proposition}[Existence of the diffusion bridge] 
\label{prop: existence_diffusion_bridge}
The mapping $h$ defined in \eqref{eq: def_h_g} satisfies Assumptions $(i)$ and $(ii)$ of Theorem \ref{thm: htransform} with $K h = 0$. Moreover, the measure $\Pstr$ defined on $\calF_T$ by
\begin{align*}
    \df \Pstr_t = \frac{h(t,X_t)}{h(0,x_0)} \df \P_t, \quad t < T,
\end{align*}
is such that, for any bounded and measurable function $\varphi$ and $0 \leq t_1 \leq ... \leq t_n < T$, it holds
\begin{align*}
    \E^{\star}[\varphi(X_{t_1},...,X_{t_n})] = \E[ \varphi(X_{t_1},...,X_{t_n}) \mid L X_T = y], \quad \calL(LX_T)-\text{a.s.}
\end{align*}
We call the process $X$ under $\Pstr$ the infinite-dimensional diffusion bridge (of $X$ given $LX_T = y$).
\end{proposition}

\begin{remark}
If, in addition to Assumption (i) and (ii) of Theorem \ref{thm: htransform}, $h$ satisfies Assumption (iii), $X$ under $\Pstr$ is a mild solution to the diffusion bridge equation 
\begin{align}
\label{eq: dXstr_t}
    \df \Xstr_t = \left[ A \Xstr_t + F(t,\Xstr_t) + Q \Df_x \log h(t, \Xstr_t) \right] \df t + Q^{\fsqrt} \df \Wstr_t, \quad t \in [0,T),
\end{align}
for some $\Pstr$-cylindrical Wiener process $\Wstr$. This generalizes the diffusion bridge equation known for conditioned diffusions in Euclidean spaces.

However, we do not impose this assumption on $h$ for two reasons. Firstly, verifying the Fréchet differentiablity of $h$ is a difficult problem in this infinite-dimensional setting. If the state space is Euclidean, this assumption can be verified based on regularity results for transition densities of diffusion processes. To the best of our knowledge, no such results are known for SPDEs. 

Secondly, except for a few special cases, the term $\Df_x \log h$ is intractable. This renders a direct simulation of the bridge process infeasible, even if it satisfies Equation \eqref{eq: dXstr_t}. This motivates the construction of the guided process as a tractable substitute process in the next proposition. 
\end{remark}

\,

\begin{proposition}[Existence of the guided process]
\label{prop: existence_guided_process}
The mapping $g$ defined in \eqref{eq: def_h_g} satisfies Assumptions $(i)$ to $(iii)$ of Theorem \ref{thm: htransform} with $K g(t,x) = \langle F(t,x), \Df_x g(t,x) \rangle.$
In particular, there exists a unique measure $\Pcrc$ on $\calF_T$, defined by 
\begin{align*}
    \df \Pcrc_t &= \dfrac{g(t,X_t)}{g(0,x_0)} \exp\left( - \int_0^t \langle F(s,X_s), \Df_x \log g(s, X_s) \rangle \, \df s \right) \df \P_t, \quad t < T,
\end{align*} 
such that $X$ under $\Pcrc$ is the unique mild solution to the SPDE
\begin{align}
\label{eq: dXcrc}
        \df \Xcrc_t = \left[ A \Xcrc_t + F(t,\Xcrc_t) + Q \Df_x \log g(t, \Xcrc_t) \right] \df t + Q^{\fsqrt} \df \Wcrc_t, \quad t \in [0,T),
\end{align} 
where $(\Wcrc_t)_{t \in [0,T)}$ is a $\Pcrc$-cylindrical Wiener process. The process $X$ under $\Pcrc$ is referred to as the guided process.
\end{proposition}

The proofs of Proposition \ref{prop: existence_diffusion_bridge} and Proposition \ref{prop: existence_guided_process} can be found in Appendix \ref{app: A}.

\begin{remark}
To highlight the fact that the process $X$ under $\Pcrc$ solves the SPDE \eqref{eq: dXcrc}, we write from now on, with slight abuse of notation, $\Xcrc$ instead of $X$ whenever we assume it to be defined on the stochastic basis $(\Om, \calF, (\calF_t)_t, \Pcrc)$.
Likewise, we write $\Xstr$ for the process $X$ defined under the measure $\Pstr$.
Furthermore, we denote by $\calLstr$ and $\calLcrc$ the law of $X$ on $C([0,T];H)$ under $\Pstr$ and $\Pcrc$ respectively and by $\calLstr_t$ and $\calLcrc_t$ their restrictions onto $C([0,t];H)$ for any $t < T$.
\end{remark}

Contrary to the function $h$, the function $g$ is available in closed form with 
\begin{align*}
    g(t,x) = \rho_Z(t,x;T,y) &= \dfrac{1}{\sqrt{(2 \pi)^k \det(R_{T-t})}} \exp \left( - \dfrac{1}{2} | R_{T-t}^{-\fsqrt} (y - L_{T-t} x) |^2 \right).
\end{align*}
Hence, the additional drift term $G(t,x) = \Df_x \log g(t,x)$ appearing in Equation \eqref{eq: dXcrc} is given by 
\begin{align}
\label{eq: G}
G(t,x) = L^*_{T-t} R^{-1}_{T-t} (y - L_{T-t} x).
\end{align}  
In particular, this enables us to draw samples directly from $\Xcrc$ by forward simulating the SPDE \eqref{eq: dXcrc}.

\section{Main results}
\label{sec: 3_main_results}
Throughout this section we fix some $T > 0$ and $y \in \R^k$ and let $\Xstr$ be the infinite-dimensional diffusion bridge and $\Xcrc$ be the guided process, respectively defined as the process $X$ under the changes of measure $\Pstr$ and $\Pcrc$ as introduced in Propositions \ref{prop: existence_diffusion_bridge} and \ref{prop: existence_guided_process}.

By construction, the measures $\Pstr$ and $\Pcrc$ are absolutely continuous on $\calF_t, t < T,$ with
\begin{align}
\label{eq: Pstrt_Pcrc_t}
        \dfrac{\df \Pstr_t}{\df \Pcrc_t}(\Xcrc) = \frac{h(t,\Xcrc_t)}{g(t,\Xcrc_t)} \frac{g(0,x_0)}{h(0,x_0)} \Psi_t(\Xcrc)  \quad \Pcrc\text{-a.s.}
\end{align}
where the process $\Psi_t(\Xcrc)$ is given by 
\begin{align}
\label{eq: Psi(t)}
\Psi_t(\Xcrc) = \exp \left( \int_0^t  \Bigl \langle F(s,\Xcrc_s),  G(s,\Xcrc_s)\Bigr \rangle \, \df s \right).
\end{align}

Our main results come in two parts. Firstly, we show that the guided process $\Xcrc(t)$ converges to the conditioning set $\{x \in H : L x = y \}$. Besides the convergence in itself, it is crucial that this convergence occurs at an appropriate rate.

Secondly, we show that the absolute continuity in \eqref{eq: Pstrt_Pcrc_t} persists as we take the limit $t \uparrow T$, i.e. that the measures $\Pstr$ and $\Pcrc$ are equivalent on $\calF_T$. As a consequence thereof, the laws $\calLstr$ and $\calLcrc$ are equivalent on the path space $C([0,T];H)$.
This is critical for the guided process to serve as a valid proposal distribution for the infinite-dimensional diffusion bridge.

\subsection{Limit behavior of the Guided Process towards the conditioning time}
\label{subsec: 3_1_limit_behavior}

We first study the limit behavior of $\Xcrc(t)$ as $t \uparrow T.$ 
Let the following hold. 
\begin{assumption}
\label{ass: Xcrc_limit_behavior}
There exist positive constants $\underbar{c}, \bar{c}$ such that for all $t \in (0,T]$
\begin{align}
\begin{split}
\label{eq: ass_Xcrc_limit_behavior}
\underbar{c} ~ t^{-1} \leq \|R_{t}^{-1}\| &\leq \bar{c} ~t^{-1}.
\end{split}
\end{align}
\end{assumption}

We get the following convergence result for the guided process $\Xcrc$. 
\begin{theorem} 
\label{thm: Xcrc_limit_behavior}
Under Assumptions \ref{ass: basic_assumptions_SPDE} and \ref{ass: Xcrc_limit_behavior}, there exists a constant $M > 0$ such that 
\begin{align}
\label{eq: Xcrc_limit_behavior}
    \limsup_{t \uparrow T} \frac{| y- L_{T-t} \Xcrc_t|}{\sqrt{(T-t) \ln(1/(T-t))}} \leq M \quad \Pcrc\text{-a.s.}
\end{align}

\end{theorem}

\subsection{Absolute continuity of \texorpdfstring{$\Pstr$}{P star} and \texorpdfstring{$\Pcrc$}{P circ} }
\label{subsec: 3_2_abs_cont}
In this section we establish sufficient conditions for the absolute continuity of $\Pstr$ with respect to $\Pcrc$ on $\calF_T$. 
We require the following assumption on the functions $h$ and $g$.

\begin{assumption}
\label{ass: absolute_continuity}
There exist continuous, positive functions $\lambda_1(t) \leq \lambda_2(t)$ satisfying $\lim_{t \to 0} \lambda_1(t) = \lim_{t \to 0} \lambda_2(t) = 1$ such that
\begin{align}
\label{eq: boundsghassumption}
\lambda_1(T-t) g(t,x) \leq h(t,x) \leq \lambda_2(T-t) g(t,x), \quad t < T, x \in H.
\end{align}
\end{assumption}

\begin{theorem}
\label{thm: absolute_continuity}
Under Assumptions \ref{ass: basic_assumptions_SPDE}, \ref{ass: Xcrc_limit_behavior} and \ref{ass: absolute_continuity}, it holds that $\Pstr$ and $\Pcrc$ are absolutely continuous on $\calF_T$ with
\begin{align}
\label{eq: Phi_T}
    \dfrac{\df \Pstr}{\df \Pcrc}(\Xcrc) = \Phi_T(\Xcrc)  \quad \Pcrc\text{-a.s.},
\end{align}
where  $\Phi_T(\Xcrc) = \frac{g(0,x_0)}{h(0,x_0)} \Psi_T(\Xcrc)$ with $\Psi_t(\Xcrc)$ defined as in Equation \eqref{eq: Psi(t)}.
\end{theorem}
We give the proof of Theorem \ref{thm: absolute_continuity} in Section \ref{sec: 7_proof_abs_cont}.

\begin{remark}
    As an immediate consequence of Theorem \ref{thm: absolute_continuity} and a change of variables, it follows that the laws $\calLstr$ and $\calLcrc$ are absolutely continuous on the path space $C([0,T];H)$.
    This implies that we can obtain weighted samples of the intractable diffusion bridge $\Xstr$ by sampling from the tractable guided process $\Xcrc$ and evaluating the weights given by the Radon–Nikodym derivative \eqref{eq: Phi_T}.
    The weights contain an intractable multiplicative term $h(0,x_0)$. This term acts as a proportionality constant that cancels out in a renormalisation step of the weights. 
\end{remark}

\subsection{Remarks on assumptions}
\label{subsec: 3_3_remarks_assumptions}

Let us briefly comment on the assumptions underlying our main results.
The assumptions in \ref{ass: basic_assumptions_SPDE} are standard, with the exception of the boundedness of $F$ in Assumption \ref{ass: basic_assumptions_SPDE}(iv). It is needed for our proof of Theorem \ref{thm: absolute_continuity}. However, our numerical illustrations in Example \ref{ex: AllenCahn} indicate that this assumption is stronger than required for the absolute continuity to hold.

Assumption \ref{ass: Xcrc_limit_behavior} essentially states that the covariance matrix $R^{-1}_{T-t}$ present in the guiding term $G(t,x)$ explodes as the guided process approaches the conditioning time $t \uparrow T$. Moreover, it is crucial that this divergence occurs at a linear rate. This aligns with the intuition gained from the simplest conditioned diffusion, the one-dimensional Brownian bridge, where the equivalent term in the Brownian bridge equation is given by $(T-t)^{-1}$.

In this infinite-dimensional setting, Assumption \ref{ass: Xcrc_limit_behavior} depends through the operators $Q_t$ on the interplay between the drift $A$ and covariance operator $Q$ of Equation \eqref{eq: semilin_spde}, as well as on the observation operator $L$. 
As illustrated in Section \ref{sec: 5_examples}, this assumption can typically be validated directly in applications of interest. It may fail though for hypo-elliptic diffusions on $\R^n$, such as an integrated diffusion. In that case, the temporal regularity of the sample paths is not the same for all components of the diffusion and therefore the conclusion of Theorem \ref{thm: Xcrc_limit_behavior} won't hold. We refer to \cite{vdMeulenBierkensSchauer2020Simulation} for a more detailed discussion.

Assumption \ref{ass: absolute_continuity} dictates the functions $h$ and $g$ to be asymptotically equivalent in the limit $t \uparrow T$. For finite dimensional state spaces, this assumption can be traced back to bounds on the transition densities of $X$ and $Z$, using, for example, Aronson's inequality (cf. \cite{Aronson1967Bounds}). In our setting, such transition densities are not guaranteed to exist, making it necessary to work with $h$ and $g$ directly instead. 

The following lemma gives a sufficient condition for Assumption \ref{ass: absolute_continuity} to be satisfied. It is shown in Section \ref{subsec: boundsghproof}.
\begin{lemma} 
\label{lem: bounds_gh}
Let the non-linearity $F$ and covariance operator $Q$ be such that $\tilde{F}(t,x) = Q^{-\fsqrt}F(t,x)$ is well defined with $|\tilde{F}(t,x)| \leq C_{\tilde{F}}$ for all $t \in [0,T], x \in H$.
Then there exists a continuous function $\lambda(t) \geq 1$ satisfying $\lim_{t \to 0} \lambda(t) = 1$ such that
\begin{align}
\label{eq: bounds_ghlemma}
    \frac{g(t,x)}{\lambda(T-t)} \leq h(t,x) \leq \lambda(T-t) g(t,x), \quad t < T,\: x \in H.
\end{align}
\end{lemma}

\section{A Metropolis-Hastings algorithm for sampling infinite-dimensional diffusion bridges}
\label{sec: 4_MH_algorithm}

We present a short description on how the theoretical result of Theorem \ref{thm: absolute_continuity} can be used to simulate the infinite-dimensional diffusion bridge $\Xstr$.
This can, for example, be achieved through a simple Metropolis-Hastings sampler (MH sampler), using the law $\calLcrc$ of the guided process as a proposal distribution to draw samples $\Xstr$ of $\calLstr$ by iterating between the following steps:
\begin{itemize}
    \item[1.] Draw a sample from the proposal distribution $\Xcrc \sim \calLcrc$.
    \item[2.] Accept/reject $\Xcrc$ with acceptance probability based on the Radon–Nikodym derivative $\Phi_T(\Xcrc)$.
\end{itemize}

In practical applications, one has to numerically approximate $X$ and $\Xstr$ on some gridded domain $D \times [0,T], D \subset \R^d$.
For such computational implementations, we write $\solve(A,F,Q,W)$ for any \textit{numerical SPDE solver} that approximates the mild solution $X$ of an SPDE of the form \eqref{eq: semilin_spde} on the given grid.
We refer to \cite{Lord2014Introduction}, Chapter 10, for an introduction to the solvers used in the field. Since we are dealing with diagonalisable operators $A$ and $Q$, a natural choice for the SPDE solver is the \textit{spectral Galerkin method}.

The proposals of $\Xcrc$ can be constructed to depend on the current value of the chain based on the \textit{preconditioned Crank-Nicolson (pCN)} scheme (see \cite{Neal1999Regression}, \cite{Beskos08MCMC}, \cite{Cotter2013MCMC}) as follows.  
In step $i$, let $Z$ denote the process such that for the current value of the chain $\Xstr_i = \solve(A, F+Q \Gcrc, Q,Z)$.
The proposal in the $i$-th sampling step then takes on the form:
\begin{itemize}
    \item[(i)] Draw a Wiener process $W$, independent of $Z$;
    \item[(ii)] Set $Z^{\circ}= \sqrt{1 - \beta^2} Z + \beta W$;
    \item[(iii)] Compute $\Xcrc = \solve(A,F+Q \Gcrc,Q,Z^{\circ})$.
\end{itemize}
Here, $\beta \in (0,1]$ denotes a hyperparameter that determines the size of the pCN step. For $\beta = 1$, independent proposals of $\Xcrc$ are drawn.

We summarize the complete algorithm in the following.

\begin{algorithm}[H]
\label{alg: MH_sampler1}
\caption{MH Sampler of $\Xstr$}
\LinesNotNumbered  
    \KwIn{SPDE Parameters $A,~F, ~Q$ and $x_0$, conditioning point $y \in H$, gridded domain $D \times [0,T]$, iterations $N$, step size $\beta$}
	\KwOut{Samples $(\Xstr_i)_{i=0}^N$ of $\Xstr$ on $D \times [0,T]$}
	\textbf{Initialize:} Draw a Wiener process $Z$ and set $\Xstr_0 = \solve(A,F+Q \Gcrc,Q,Z)$\;
	\For{$i = 0... ~N-1$}{
 	\textbf{Proposal} \\
                \quad(i) ~Draw a Wiener process $W$\; 
  			\quad(ii) Set $Z^{\circ}= \sqrt{1 - \beta^2} Z + \beta W$\;
            \quad(iii) Compute $\Xcrc = \solve(A,F+Q \Gcrc,Q,Z^{\circ})$\;
        \textbf{Update}  \\
  		\quad(i) ~Compute $M = \min\left(1, \dfrac{\Phi_T(\Xcrc)}{\Phi_T(\Xstr_i)} \right)$\;
            \quad(ii) Draw $U \sim \text{Unif}(0,1)$\; 
        \quad \quad \If{$U < M $}{
      \quad \quad Set $\Xstr_{i+1} = \Xcrc$ and $Z = Z^{\circ}$\;
    }
    \quad \quad \Else{
       \quad \quad Set $\Xstr_{i+1} = \Xstr_i$.
    }
  	}
\end{algorithm}

\begin{remark}
Since the results in Theorem \ref{thm: absolute_continuity} are infinite-dimensional, the validity of Algorithm \ref{alg: MH_sampler1} depends neither on the choice of the numerical solver nor on the discretisation of the domain $D \times [0,T]$.
In particular, for a practical implementation, both the solver and the mesh size can be freely chosen and the implementation will remain valid if the mesh size tends to zero. 
With a trade-off of greater computational costs, a finer mesh size will not only improve the approximation quality of the numerical SPDE solver, but also the approximation of the likelihood ratios $\Phi_T(\Xcrc)$. 
\end{remark}

\section{Examples}
\label{sec: 5_examples}

\subsection{Diagonalisable equations}
Consider the case that the linear operators $A$ and $Q$ share an eigenbasis of $H$. Specifically, assume the following. 
\begin{assumption}
\label{ass: AQdiagonal}
    The operators $A$ and $Q$ are diagonalisable, i.e. there exists an orthonormal basis $(e_j)_{j \in \N}$ of $H$ and positive sequences $(a_j)_{j \in \N}$ and $(q_j)_{j \in \N}$ such that $a_j \to \infty$, $\sup q_j < \infty$ and
    \begin{align}
    \begin{split}
        A e_j &= -a_j e_j  \\
        Q e_j &= q_j e_j.
    \end{split}
    \end{align}
\end{assumption}
Given Assumption \ref{ass: AQdiagonal}, Equation \eqref{eq: semilin_spde} is commonly referred to as a \textit{diagonalisable SPDE}.
If $F$ is a linear operator that commutes with $A$ and $Q$, such an equation is referred to as \textit{fully diagonalisable}. In that case, the mild solution $X$ is an Ornstein-Uhlenbeck process that can be represented as an infinite series of decoupled SODEs in its eigenmodes. Note, however, that such a decomposition is no longer valid when $F$ is a non-linear operator. 

In the diagonalisable setting, $A$ generates a strongly continuous semigroup $(S_t)_{t \geq 0}$ with $S_t e_j = \exp(-a_j t) e_j$ and, in particular, satisfies Assumption \ref{ass: basic_assumptions_SPDE}(i). Moreover, Assumption \ref{ass: basic_assumptions_SPDE}(iii) is equivalent to 
\begin{align}
\label{eq: q_ja_jbound}
    \sum_{j=1}^{\infty} \dfrac{q_j}{a_j} < \infty,
\end{align}
under which the operators $(Q_t)_{t \geq 0}$ are diagonalisable with eigenvalues
\begin{align}
\begin{split}
\label{eq: eigenvalues_Qt}
q_j(t) &= \frac{q_j}{2 a_j} \left[ 1 - \exp(-2 a_j t) \right].
\end{split}
\end{align}

In the following proposition, we show that, given a diagonalisable SPDE, Assumption \ref{ass: Xcrc_limit_behavior} is satisfied for two common choices of the observation operator $L$.

\begin{proposition}
\label{prop: LPkPw}
Let Assumption \ref{ass: AQdiagonal} and Equation \eqref{eq: q_ja_jbound} be satisfied.
Then, Assumption \ref{ass: Xcrc_limit_behavior} holds in either of the following cases:
\begin{itemize}
    \item $L = P_k$ is the projection onto the first $k$ coordinates with respect to the eigenbasis $(e_j)_{j \in \N}$, i.e. \begin{align}
    \label{eq: Pk}
        P_k x = \left( \langle x, e_j \rangle \right)_{j=1}^k, \quad x \in H.
    \end{align}
    \item $L = P_w$ is the projection onto the coordinate of the subspace $\text{span}\{w\}$ for some fixed $w \in H$, i.e.
    \begin{align}
    \label{eq: L_w}
        P_w x = \langle x, w \rangle, \quad x \in H.
    \end{align}
\end{itemize}
\end{proposition}
\begin{proof}
Since $R_t = L Q_t L^*$ is symmetric positive definite, to show \eqref{eq: ass_Xcrc_limit_behavior} it suffices to show that
\begin{align}
\label{eq: ap192873132}
\begin{split}
    \lambda_{\max}(R_t) &\leq t ~ \underbar{c}^{-1}, \\
    \lambda_{\min}(R_t) &\geq t ~ \bar{c}^{-1},  \quad t \in (0,T],
\end{split} 
\end{align}
for some constants $\bar{c} > 0$ and $\underbar{c} > 0$, with $\lambda_{\min}(R_t)$ and $\lambda_{\max}(R_t)$ denoting the smallest and largest eigenvalue of $R_t$.
We will use that, for any $j \geq 1$, 
\begin{align}
\begin{split}
\label{eq: ao123798721o3}
\sup_{t \in (0,T]} t^{-1} (1 - \exp(-2 a_j t)) &= 2 a_j \\
\inf_{t \in (0,T]} t^{-1} (1 - \exp(-2 a_j t)) &= T^{-1} (1 - \exp(-2 a_j T)),
\end{split}
\end{align} 
which follows from a Taylor expansion at $t = 0$. 

Consider the first case of $L = P_k$. A simple computation shows that $R_t$ is given by the diagonal $k \times k$-matrix with entries $q_j(t), j = 1,...,k$. Hence, plugging in $q_j(t)$ as in \eqref{eq: eigenvalues_Qt} and \eqref{eq: ao123798721o3} we have
\begin{align*}
\begin{split}
\max_{1 \leq j \leq k} \sup_{t \in (0,T]} t^{-1} q_j(t) &= \max_{1 \leq j \leq k} \sup_{t \in (0,T]}  \frac{q_j}{2 a_j} t^{-1} (1 - \exp(-2 a_j t)) \\
&= \max_{1 \leq j \leq k} q_j =: \underbar{c}^{-1}
\end{split}
\end{align*}
as well as
\begin{align*}
\begin{split}
\min_{1 \leq j \leq k} \inf_{t \in (0,T]} t^{-1} q_j(t) &= \min_{1 \leq j \leq k} \inf_{t \in (0,T]}  \frac{q_j}{2 a_j} t^{-1} (1 - \exp(-2 a_j t)) \\
&= \min_{1 \leq j \leq k}  \frac{q_j}{2 a_j} T^{-1} (1 - \exp(-2 a_j T)) =: \bar{c}^{-1}.
\end{split}
\end{align*}
This shows \eqref{eq: ap192873132} in the first case. For the second case, note that $P_w^*$ is the mapping $y \mapsto y w \in H$ for any $y \in \R$. From this, it follows that 
\begin{align*}
    R_t y = (P_w Q_t P_w^*)y = (L Q_t) (yw) = \langle Q_t w, w \rangle y
\end{align*}
and hence $R_t = \langle Q_t w, w \rangle$. Plugging in $q_j(t)$ as given in \eqref{eq: eigenvalues_Qt}, it holds
\begin{align*}
    \sup_{t \in (0,T]} t^{-1} R_t &= \sup_{t \in (0,T]} t^{-1} \langle Q_t w, w \rangle \\
    &\leq \sum_{j=1}^{\infty} \frac{q_j}{2 a_j} \left( \sup_{t \in (0,T]} t^{-1} \left( 1 - \exp(-2 a_j t) \right) \right) |\langle w, e_j \rangle|^2 \\
    &= \sum_{j=1}^{\infty} q_j |\langle w, e_j \rangle|^2 = \langle Q^{\fsqrt} w, Q^{\fsqrt} w \rangle =: \underbar{c}^{-1}.
\end{align*}
Here, we used \eqref{eq: ao123798721o3} in the third step and the fact that $Q$ is positive in the last step. 
This shows the first inequality in \eqref{eq: ap192873132}. The second inequality follows from the fact that 
\begin{align*}
    \inf_{t \in (0,T]} t^{-1} R_t &\geq \sum_{j=1}^{\infty} \frac{q_j}{2 a_j} \left( \inf_{t \in (0,T]} t^{-1} \left( 1 - \exp(-2 a_j t) \right) \right) |\langle w, e_j \rangle|^2 \\
    &=  \sum_{j=1}^{\infty} \frac{q_j}{2 a_j} T^{-1} \left( 1 - \exp(-2 a_j T) \right) |\langle w, e_j \rangle|^2 =: \bar{c}^{-1},
\end{align*}
where we again used \eqref{eq: ao123798721o3} in the second step and \eqref{eq: q_ja_jbound} in the last step to justify the convergence of the sum. This shows \eqref{eq: ap192873132} and thus finishes the proof.

\end{proof}

\begin{remark}
Consider the special case that $H = L^2(D)$ for some bounded domain $D \subset \R^d$. 
In that case, the observation operator
\begin{align*}
    P_w x = \langle x, w \rangle = \int_D x(\xi) w(\xi) \, \df \xi
\end{align*} 
of Proposition \ref{prop: LPkPw} can be understood as observing a weighted average $P_w X_T$ of $X_T$. Depending on the choice of $w \in L^2(D)$, this can, for example, correspond to observing a global or local average of $X_T$.
Moreover, it is clear that the second case in Proposition \ref{prop: LPkPw} extends to the case that one observes a set of weighted averages $\int_D x(\xi) w_j(\xi) \, \df \xi$ for functions $w_1,...,w_k \in L^2(D)$. 
\end{remark}

\begin{example}[Stochastic reaction-diffusion equations]
\label{ex: reaction_diffusion_equations}
Consider a stochastic reaction-diffusion equation, formally written as 
\begin{align}
\begin{split}
\label{eq: stoch_reaction_diffusion}
\partial_t X(t,\xi) &=  \partial^2_{\xi} X(t,\xi) + f(X(t,\xi)) + \dot{w}(t,\xi), \quad t,\xi \in [0,T] \times [0, \pi], 
\end{split}
\end{align}
with initial condition $X(0,\xi) = x_0(\xi)$ and homogeneous Dirichlet boundary conditions $X(t,0) = X(t,\pi) = 0$ for all $t, \xi \in [0,T] \times [0,\pi]$.
Here, $f$ is a continuous, scalar-valued function of at most linear growth and $\dot{w}$ is a Gaussian noise term that is white in time and possibly spatially dependent. 

It is well known that \eqref{eq: stoch_reaction_diffusion} can be stated as a semilinear SPDE of the form \eqref{eq: semilin_spde}, where $A$ denotes the Dirichlet Laplace operator $A = \partial^2_{\xi}$ on $H = L^2([0,\pi])$ with domain $D(A) = H^2([0,\pi]) \cap H^1_0([0,\pi])$. Here, $H^2([0,\pi])$ is the Sobolev space $W^{2,2}([0,\pi])$ and $H^1_0$ denotes the closure of the smooth and compactly supported test functions in $W^{1,2}([0,\pi])$.
Furthermore, $F$ is the nonlinearity given by the Nemytskii operator $F(X_t)(\xi) = f(X(t,\xi))$ generated by $f$ and $Q$ is a symmetric, positive definite operator on $H$. 

The Dirichlet Laplace operator admits an eigenbasis $e_j(\xi) = \sin( j \xi)$ with eigenvalues $- a_j = - j^2, j \geq 1$ on $H$.
Hence, Assumption \ref{ass: AQdiagonal} and Equation \eqref{eq: q_ja_jbound} are satisfied for any operator $Q$ on $H$ that is diagonalisable with respect to $(e_j)_j$ with eigenvalues $q_j \sim j^{-r}$ for some $r \geq 0.$
\end{example}

\subsection{Stochastic Amari equation}
Assume the stochastic Amari equation, formally given by
\begin{align}
\begin{split}
\label{eq: stoch_amari_equation}
\partial_t X(t,\xi) &=  - X(t,\xi) + \int_{[0,\pi]} f(\xi,\xi') s(X(t,\xi')-\theta) \, \df \xi' + \dot{w}(t,\xi), \quad t,\xi \in [0,T] \times [0, \pi],
\end{split}
\end{align}
with initial condition $X(0, \xi) = x_0(\xi)$.
Here, $X(t,\xi)$ models the activity of a neural field at time $t$ and location $\xi$. The mapping $f$ represents a spatial connectivity function on $[0,\pi]$, whereas $s$ is an activation function with threshold $\theta > 0$. 
The term $-X(t,\xi)$ is a linear damping term that models relaxation of neural activity.

For more details on the Amari equation and neural field models we refer to the recent survey in \cite{Cook2022Neural} and to \cite{Faugeras2015Stochasticneuralfield} for a mathematical view on the field. 

The Amari equation can be formulated as a semilinear SPDE of the form \eqref{eq: semilin_spde} on $H = L^2([0,\pi])$ with $A X = - X$ and $F(X)(\xi) = \int_{[0,\pi]} f(\xi,\xi') s(X(\xi')-\theta) \df \xi'$. Typically, $f$ is chosen to be continuous and $s$ to be Lipschitz continuous, in which case $F$ is itself a bounded and Lipschitz continuous mapping on $H$. 
Since the semigroup generated by $A$ is no longer smoothing, one requires the noise operator $Q$ to be trace-class. 
In that case, Assumption \ref{ass: basic_assumptions_SPDE} is met. Moreover, it follows just as in the proof of Proposition \ref{prop: LPkPw}, that Assumption \ref{ass: Xcrc_limit_behavior} is satisfied for the observation operators $L = P_k$ and $L = P_w$.

We apply Lemma \ref{lem: bounds_gh} to show that Assumption \ref{ass: absolute_continuity} is satisfied for common choices of the spatial connectivity function $f$. 
One such typical choice is given by a difference of Gaussian kernels 
\begin{align*}
    f(\xi,\xi') &= A_1 \exp \left(- \dfrac{|\xi - \xi'|^2}{2 \sigma_1^2} \right) - A_2 \exp \left(- \dfrac{|\xi - \xi'|^2}{2 \sigma_2^2} \right)
\end{align*}
for some positive parameters $A_1, A_2, \sigma_1, \sigma_2$.
Consider the associated Hilbert-Schmidt integral operator $T_f$ defined by 
\begin{align*}
    (T_f u)(\xi) =  \int_{[0,\pi]} f(\xi,\xi') u(\xi') \df \xi', \quad u \in H.
\end{align*}
By Mercer's theorem, it follows that $T_f$ admits an eigenbasis $(e_j)_{j \geq 1}$ of $H$ and it is a well-known result that the eigenvalues $\lambda_j$ of $T_f$ are of exponential decay $\lambda_j \sim \exp(-j^2)$.
Hence, representing the non-linearity $F$ as $F(X) = T_f s_{\theta}(X)$, where we write with slight abuse of notation $s_{\theta}$ for the Nemytskii operator $s_{\theta}(X)(\xi) = s(X(\xi)-\theta)$, it follows that for any trace-class operator $Q$ with eigenvalues $q_j \sim j^{-r}, r > 1$, we have that $Q^{-1/2} F = Q^{-1/2} T_f s_{\theta}$ is well-defined and bounded.

\begin{figure}[ht]
    \centering
    \subfigure{
        \includegraphics[width=0.47\linewidth]{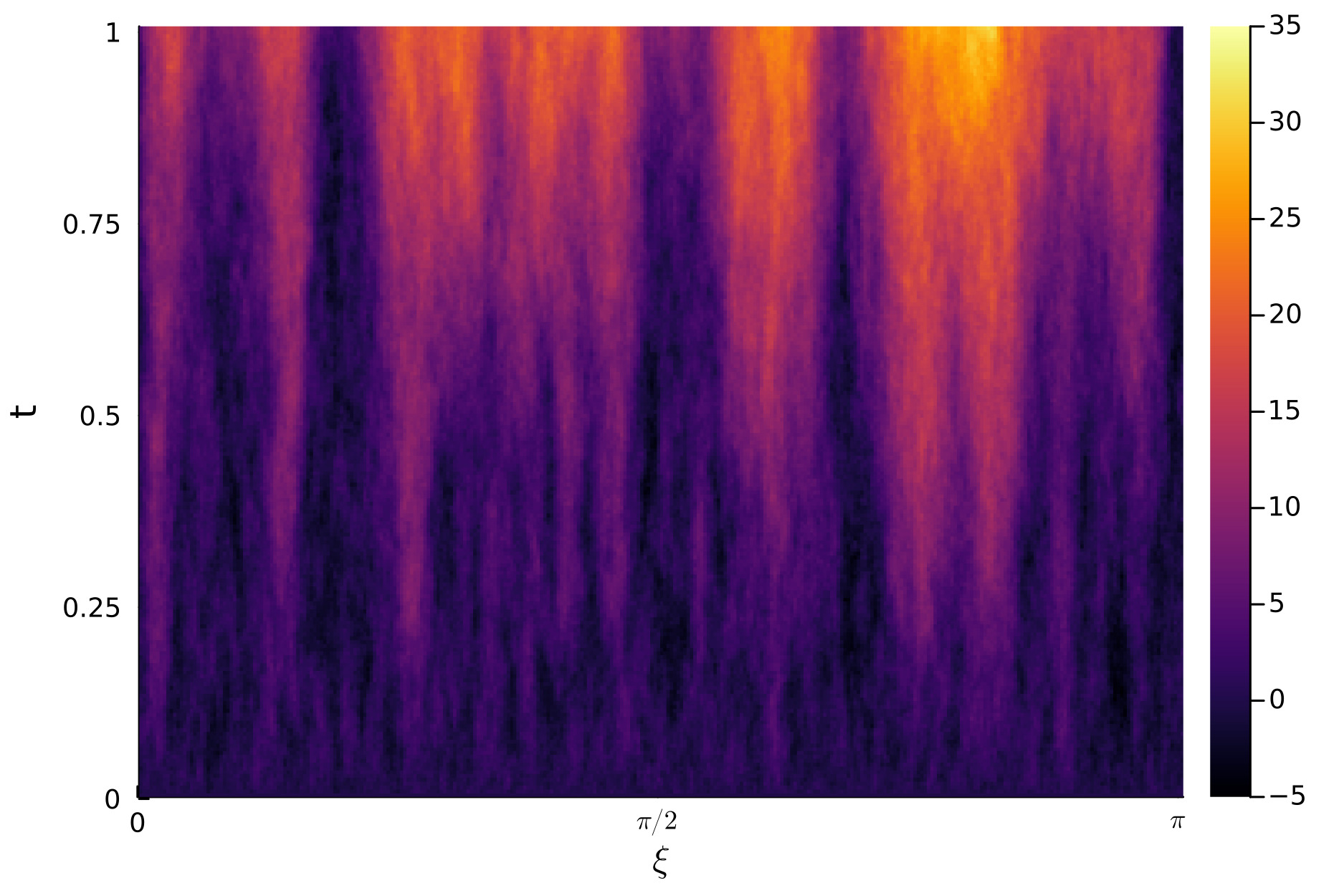}
    }
    \hfill
    \subfigure{
        \includegraphics[width=0.47\linewidth]{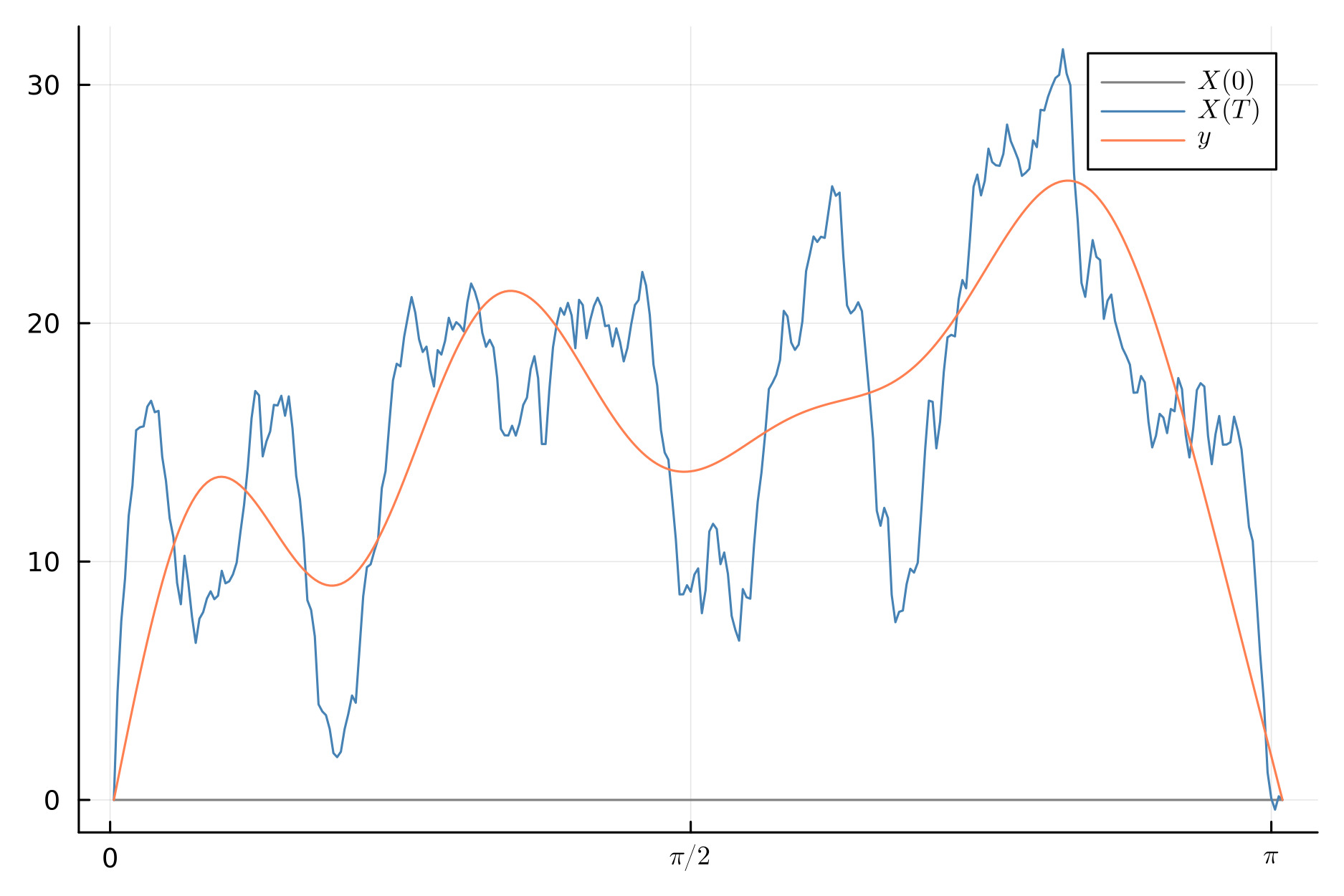}
    }
    \caption{Forward simulation of $X$. Left: Heatmap of a sample path $X(t,\xi)$ of Equation \eqref{eq: MichaelisMenten} with parameters as specified in \eqref{eq: MM_params}. Right: The states $X_0$ and $X_T$ and the observation $y = P_k X_T$, $k=10$. }
    \label{fig: MM_forward}
\end{figure}

\subsection{Numerical illustrations}

We illustrate the performance of Algorithm \ref{alg: MH_sampler1} with numerical examples for two distinct stochastic reaction-diffusion equations.

\begin{example}
\label{ex: MichaelisMenten}
Assume the SPDE given by
\begin{align}
\label{eq: MichaelisMenten}
            \df X_t &= \left[  \eta A X_t + \dfrac{ \zeta_1 X^2_t}{1 + \zeta_2 X^2_t} \right] \df t + \sigma \df W_t,~ X(0) = x_0,
\end{align}
with Dirichlet Laplacian $A$, white noise covariance operator $Q = \sigma^2 \text{Id}$ and non-linearity $F$ generated by
\begin{align}
\label{eq: MM_nonlinearity}
f(x) = \dfrac{\zeta_1 x^2}{1 + \zeta_2 x^2}.
\end{align}

Reactions terms of the form \eqref{eq: MM_nonlinearity} are present for example in predator-prey models (\cite{Holling1959Components}, \cite{Dawes2013Derivation}), activator-inhibitor models (\cite{Young1984Local}, \cite{Pasemann2021Diffusivity}) or Michaelis-Menten kinetics (\cite{Johnson2011Original}).
Typically, the given examples model multiple interacting populations with a system of coupled SPDEs and non-linearity \eqref{eq: MM_nonlinearity} depending linearly on another population. 
However, for the sake of simplicity, we restrict our attention to the case where such additional populations are assumed to be constant, leading to a reaction-diffusion equation of the form \eqref{eq: MichaelisMenten}.

Following Example \ref{ex: reaction_diffusion_equations}, Equation \eqref{eq: MichaelisMenten} is diagonalisable such that the condition \eqref{eq: q_ja_jbound} is met. Moreover, since $F$ is bounded and $Q$ is boundedly invertible, it follows that both Assumption \ref{ass: basic_assumptions_SPDE} and Assumption \ref{ass: absolute_continuity} are satisfied.
We assume one partial observation $y \in \R^k$ given at time $T$ by the spectral projection $y = P_k X_T$ as defined in \eqref{eq: Pk}. By Proposition \ref{prop: LPkPw}, Assumption \ref{ass: Xcrc_limit_behavior} is then also satisfied. 

For our numerical implementation, we set the parameters of Equation \eqref{eq: MichaelisMenten} as 
\begin{align}
\label{eq: MM_params}
[\eta, \zeta_1, \zeta_2, \sigma] = [3 \times 10^{-3}, 3, 0.1, 1].
\end{align}
Figure \ref{fig: MM_forward} shows the heatmap of one sample path of $X$ on the time interval $[0,T] = [0,1]$ with initial value $x_0 \equiv 0$. The path is sampled based on a spectral Galerkin approximation, using the first $100$ eigenfunctions, and a semi-implicit Euler-Maruyama scheme to approximate the resulting SODE with time steps $\Delta t = 0.005$.
For the observation $y = P_k X_T$ we set $k = 10$. 
The initial state $x_0$, true state $X_T$ and observed state $y$ are shown on the right-hand side in Figure \ref{fig: MM_forward}.

\begin{figure}[ht]
    \centering
    \includegraphics[width=0.5\textwidth]{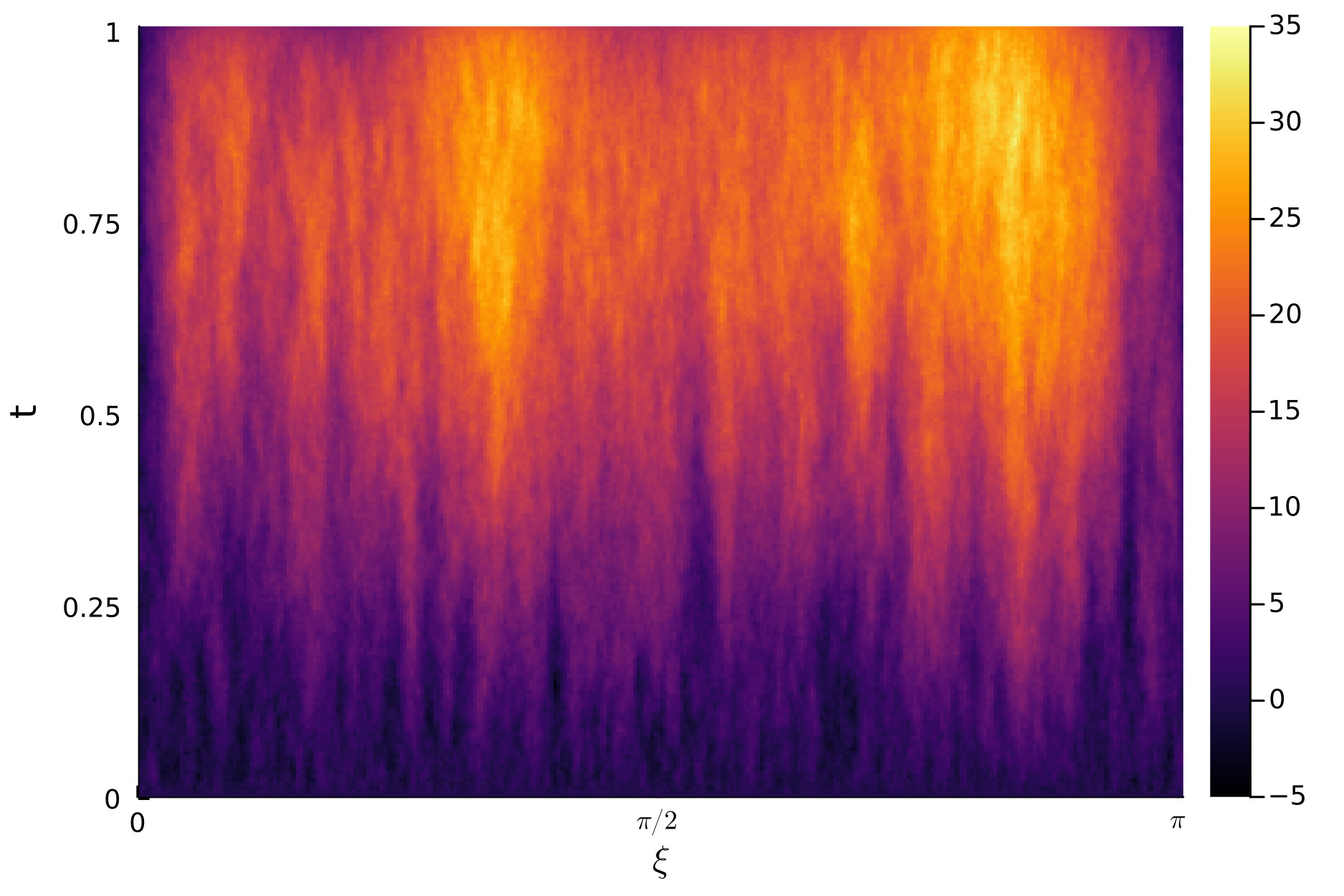} 
    \caption{Heatmap of a sample path $\Xcrc(t,\xi)$ of the guided process corresponding to \eqref{eq: MichaelisMenten} with conditioning state $y = P_k X_T, k = 10$.}
    \label{fig: MM_guided}
\end{figure}

\begin{figure}[ht]
    \centering
    \includegraphics[width=0.5\textwidth]{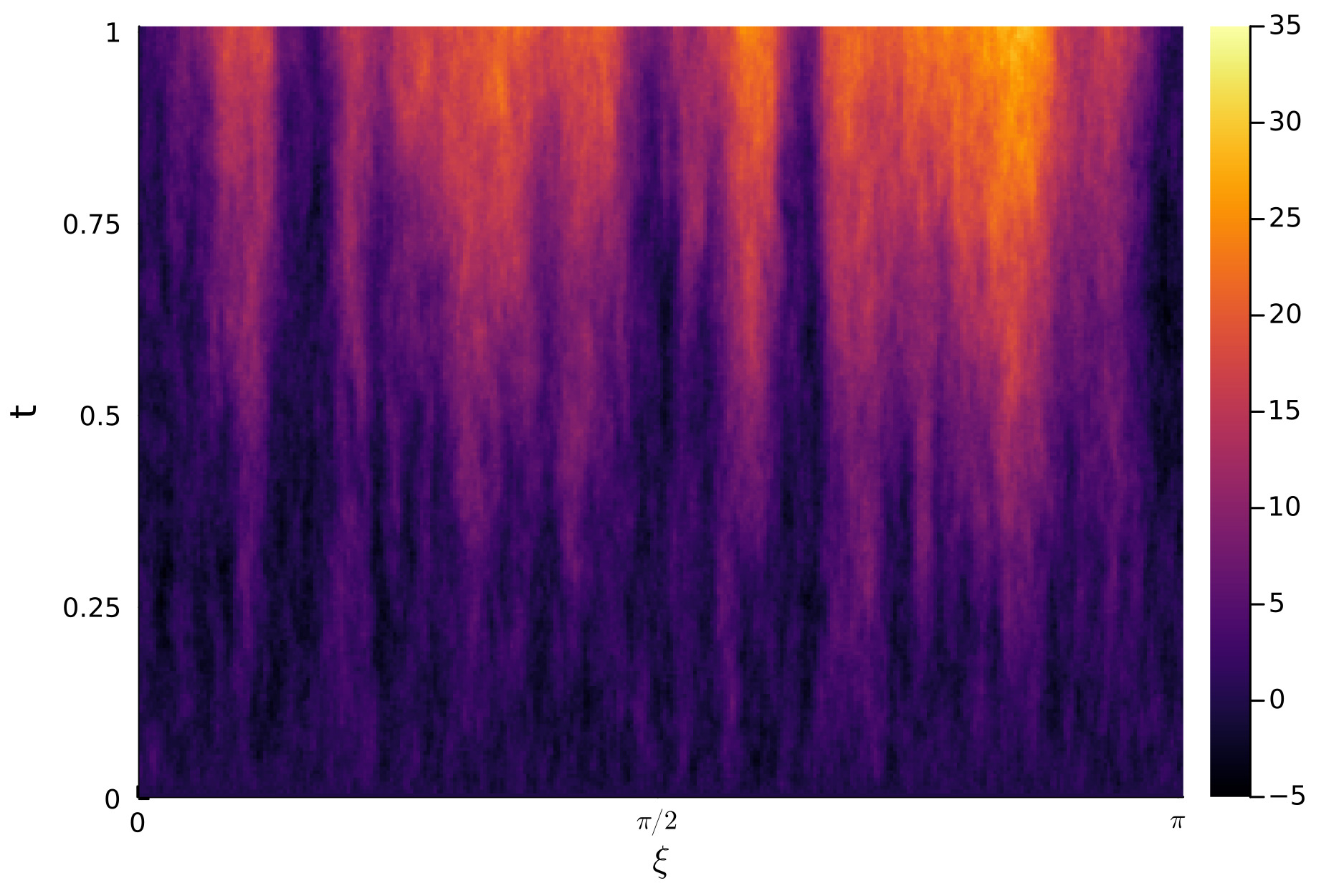} 
    \caption{Heatmap of the mean sample path of $100$ samples of the estimated diffusion bridge $\Xstr(t,\xi)$ of Equation \eqref{eq: MichaelisMenten} returned by Algorithm \ref{alg: MH_sampler1}.}
    \label{fig: MM_bridge}
\end{figure}

We employ Algorithm \ref{alg: MH_sampler1} to sample from the infinite-dimensional diffusion bridge $\Xstr$ of Equation \eqref{eq: MichaelisMenten}, conditioned on the observed state $y$.
In Figure \ref{fig: MM_guided}, we display a random sample path of the guided process $X^{\circ}$, simulated with the same SPDE solver and mesh size as the path of Figure \ref{fig: MM_forward}. 
It serves as the proposal distribution in the MH sampler. While indeed the process is forced to satisfy $L \Xcrc_t = y$, one can visually assess that the guided process looks different from the forward simulated path in Figure \ref{fig: MM_forward}, motivating the application of Algorithm \ref{alg: MH_sampler1}.

The MH sampler is run with step size $\beta = 0.08$ and $50 ~ 000$ iterations. The resulting acceptance rate of proposals equals $26 \%$.
Figure \ref{fig: MM_bridge} shows the mean sample path of the last $100$ samples of the estimated diffusion bridge distribution.
A visual comparison with the sample paths of the data generating process $X$ and the guided process $\Xcrc$ strongly suggests that Algorithm \ref{alg: MH_sampler1} correctly samples from the distribution of the diffusion bridge $\Xstr$.

\begin{figure}[ht]
    \centering
    \includegraphics[width=0.8\textwidth]{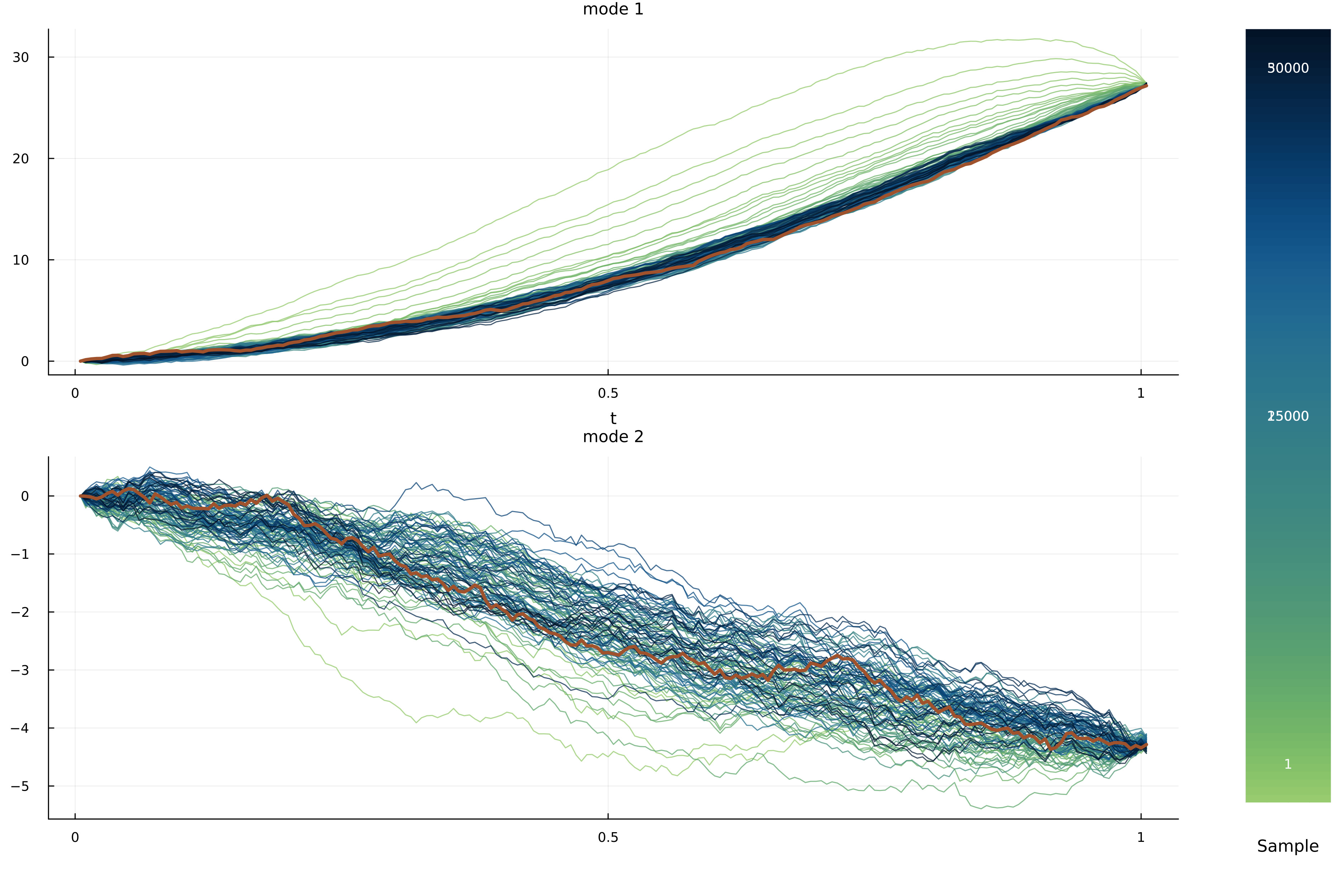} 
    \caption{Paths of two spectral modes of the diffusion bridge samples of Equation \eqref{eq: MichaelisMenten} returned by Algorithm \ref{alg: MH_sampler1}. Every $500$-th sample is shown. Green paths represent `earlier' samples in the Markov chain, whereas blue paths represent `later' samples. The orange path is the spectral mode of the data generating process $X$.}
    \label{fig: MM_spectralbridges}
\end{figure}

\begin{figure}[ht]
    \centering
    \includegraphics[width=0.8\textwidth]{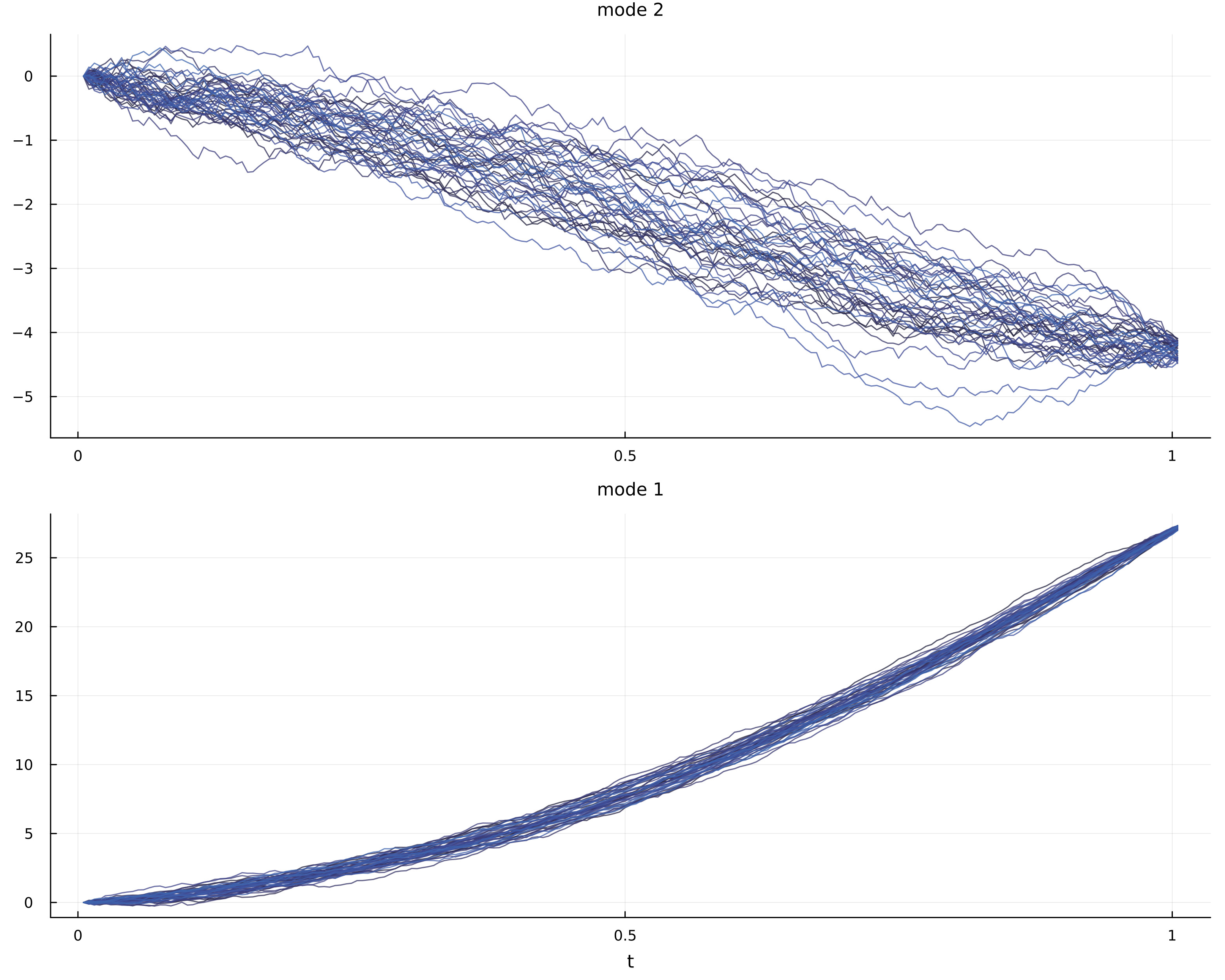} 
    \caption{Sample paths in the spectral modes $1$ and $2$ obtained by forward sampling the SPDE \eqref{eq: MichaelisMenten} with parameters as specified in \eqref{eq: MM_params}. Only paths that satisfy \eqref{eq: MMexperimentsconditioning} are kept.}
    \label{fig: MM_spectralforward}
\end{figure}

This is further supported by the display of the sampler's performance in Figures \ref{fig: MM_spectralbridges} and \ref{fig: MM_spectralforward}. For this, we let 
\begin{align}
\label{eq: j1j2}
    j_1 &= \arg \max_j x_j(T) = \arg \max_j \langle X_T, e_j \rangle, \\
    j_2 &= \arg \min_j x_j(T) = \arg \min_j \langle X_T, e_j \rangle
\end{align}
be the indices of the largest and smallest spectral mode of $X_T$. In our example, this corresponds to $j_1 = 1$ and $j_2 = 2$. 
Out of the $50 ~ 000$ sampled paths, Figure \ref{fig: MM_spectralbridges} shows every $500$-th path of the spectral modes $\{\xstr_{j_k}(t): t \in [0,T]\}, k = 1,2$, as well as the `true' path of the data generating process $x_j(t)$. 

For a comparison, we plot in Figure \ref{fig: MM_spectralforward} the `typical' sample path of the diffusion `bridge' conditioned on hitting $x_{j_k}(T), k =1,2$. 
This is achieved by forward sampling $100 ~ 000$ paths $Y$ of Equation \eqref{eq: MichaelisMenten} and only keeping those sample paths that satisfy 
\begin{align}
\label{eq: MMexperimentsconditioning}
    |y_{j_k}(T) - x_{j_k}(T)| < \varepsilon, \quad k = 1,2.
\end{align} Here we choose $\varepsilon = 0.2$. 
Note that the other spectral modes are disregarded. Even still, due to the high dimensionality of the problem, this is a rare event - from the $100 ~ 000$ forward samples all but $47$ samples are discarded. A comparison between Figures \ref{fig: MM_spectralbridges} and \ref{fig: MM_spectralforward} shows that the sample paths resemble each other well.
\end{example}

\begin{figure}[ht]
    \centering
    \subfigure{
        \includegraphics[width=0.47\linewidth]{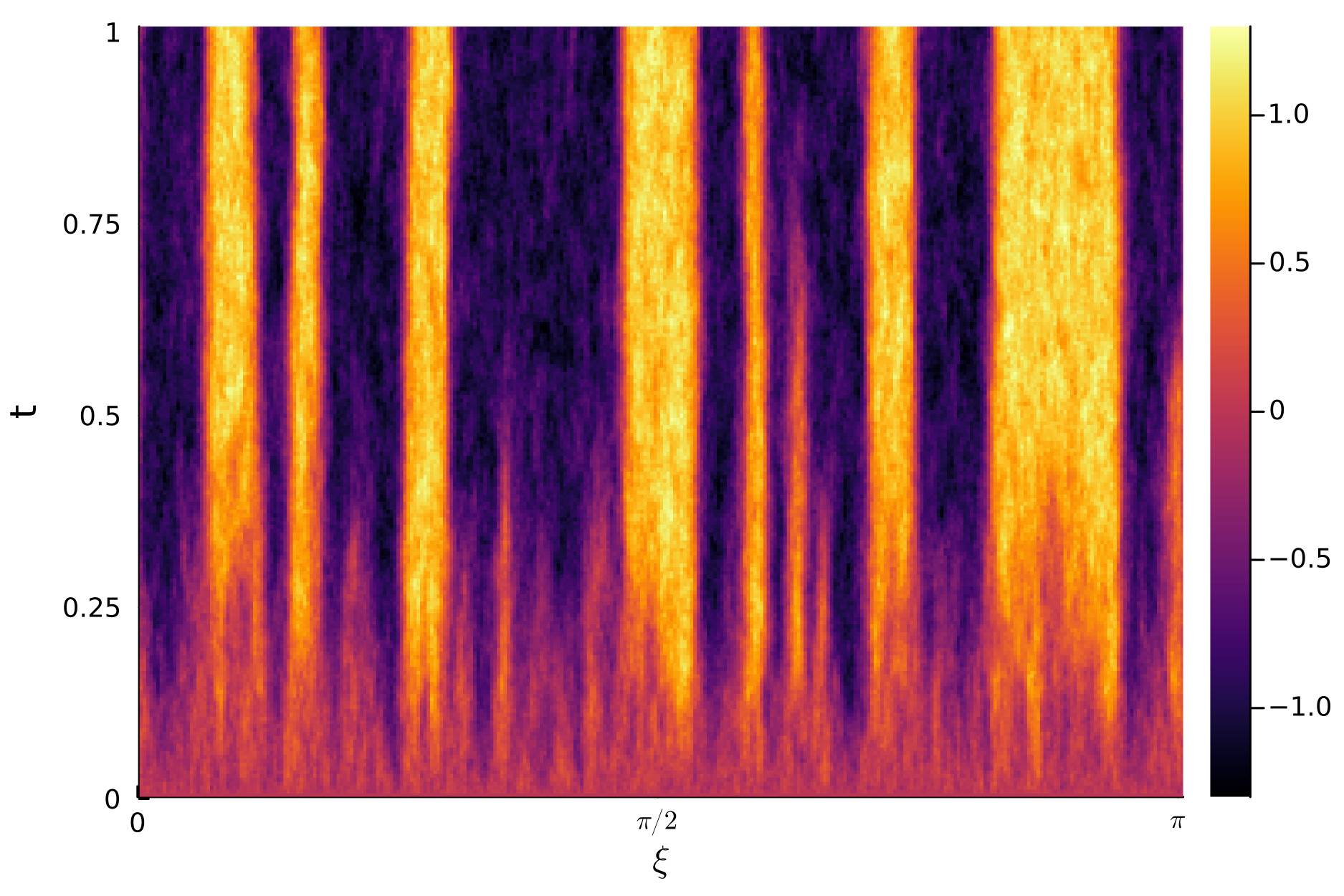}
    }
    \hfill
    \subfigure{
        \includegraphics[width=0.47\linewidth]{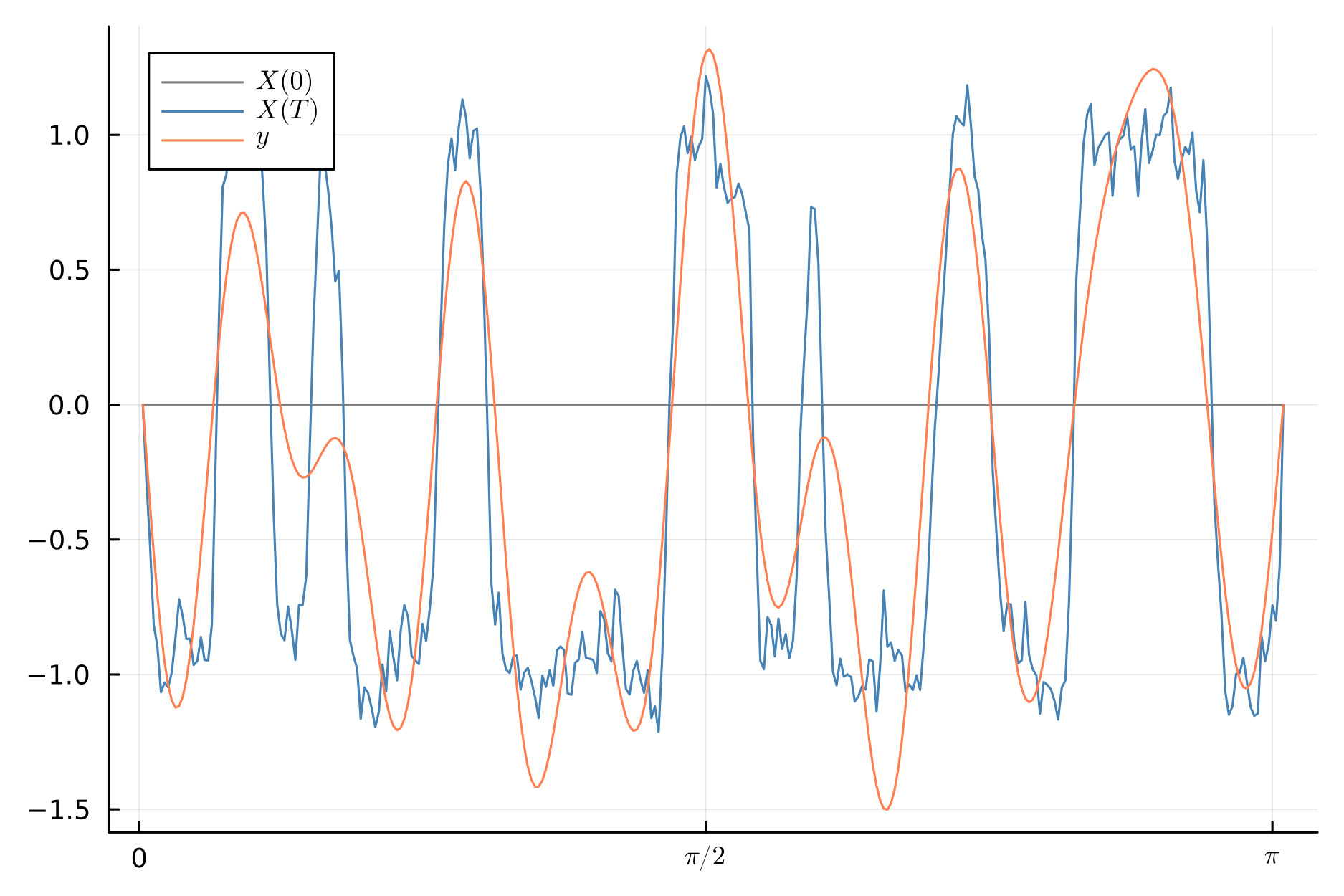}
    }
    \caption{Forward simulation of $X$. Left: Heatmap of a sample path $X(t,\xi)$ of Equation \eqref{eq: AllenCahn} with parameters as specified in \eqref{eq: AC_params}. Right: The states $X_0$ and $X_T$ and the observation $y = P_k X_T$, $k=20$. }
    \label{fig: AC_forward}
\end{figure}

\begin{example}[Stochastic Allen-Cahn Equation]
\label{ex: AllenCahn}
Consider the stochastic Allen-Cahn equation 
\begin{align}
\label{eq: AllenCahn}
            \df X_t &= \left[  \eta A X_t + \zeta (X_t - X^3_t) \right] \df t + Q_{\sigma}^{\fsqrt} \df W_t,~ X(0) = x_0,
\end{align}
with Dirichlet Laplacian $A$, non-linearity $F$ generated by $f(x) = x - x^3$, diffusion parameter $\eta > 0$ and reaction rate $\zeta > 0$.
Here, we choose a diagonalisable trace-class operator $Q_{\sigma}$ with eigenvalues 
\begin{align*}
    q_j = \sigma_0^2 (\rho^{-2} + (2 \pi j)^2)^{-(1/2 + \nu)}
\end{align*}
parametrized by $\sigma = (\sigma_0, \rho, \nu).$ 
The operator $Q_{\sigma}$ then corresponds to the Hilbert-Schmidt integral operator 
\begin{align*}
    (Q_{\sigma}x)(\xi) = \int_{[0, \pi]} q(|\xi - \xi'|) x(\xi') \, \df \xi', \quad x \in L^2([0,\pi]),
\end{align*}
where $q(r)$ is the Matérn covariance function $q(r) = \sigma_0^2 \rho^{-1} r K_{\nu}(r/\rho)$, with $K_{\nu}$ denoting the second kind modified Bessel function of order $\nu$. This follows from the fact that a Gaussian variable on $H$ with covariance operator $Q_{\sigma}$ is itself the weak solution to a stochastic fractional Laplace equation, see for example \cite{Borovitskiy2020Matern}, Theorem 5. 
The parameter $\sigma_0$ corresponds to the marginal variance of $q(r)$, whereas $\rho$ is a precision parameter. The parameter $\nu$ controls the `smoothness' of $Q_{\sigma}$. Specifically, it can be shown that a Gaussian variable on $H$ with covariance operator $Q_{\sigma}$ is $\nu$-times mean differentiable.

Just as in Example \ref{ex: MichaelisMenten}, $A$ and $Q_{\sigma}$ satisfy the Assumptions \ref{ass: basic_assumptions_SPDE} and \ref{ass: Xcrc_limit_behavior}.
However, since $F$ is not bounded, Assumption \ref{ass: basic_assumptions_SPDE}(iv) is not met. Moreover, the unbounded nature of $F$ and the fact that $Q_{\sigma}$ is trace-class renders Assumption \ref{ass: absolute_continuity} difficult to verify as Lemma \ref{lem: bounds_gh} is no longer applicable. 
Nonetheless, the results of our numerical experiments indicate that these assumptions are stricter than necessary for the results of Theorem \ref{thm: absolute_continuity} to hold.

\begin{figure}[ht]
    \centering
    \includegraphics[width=0.5\textwidth]{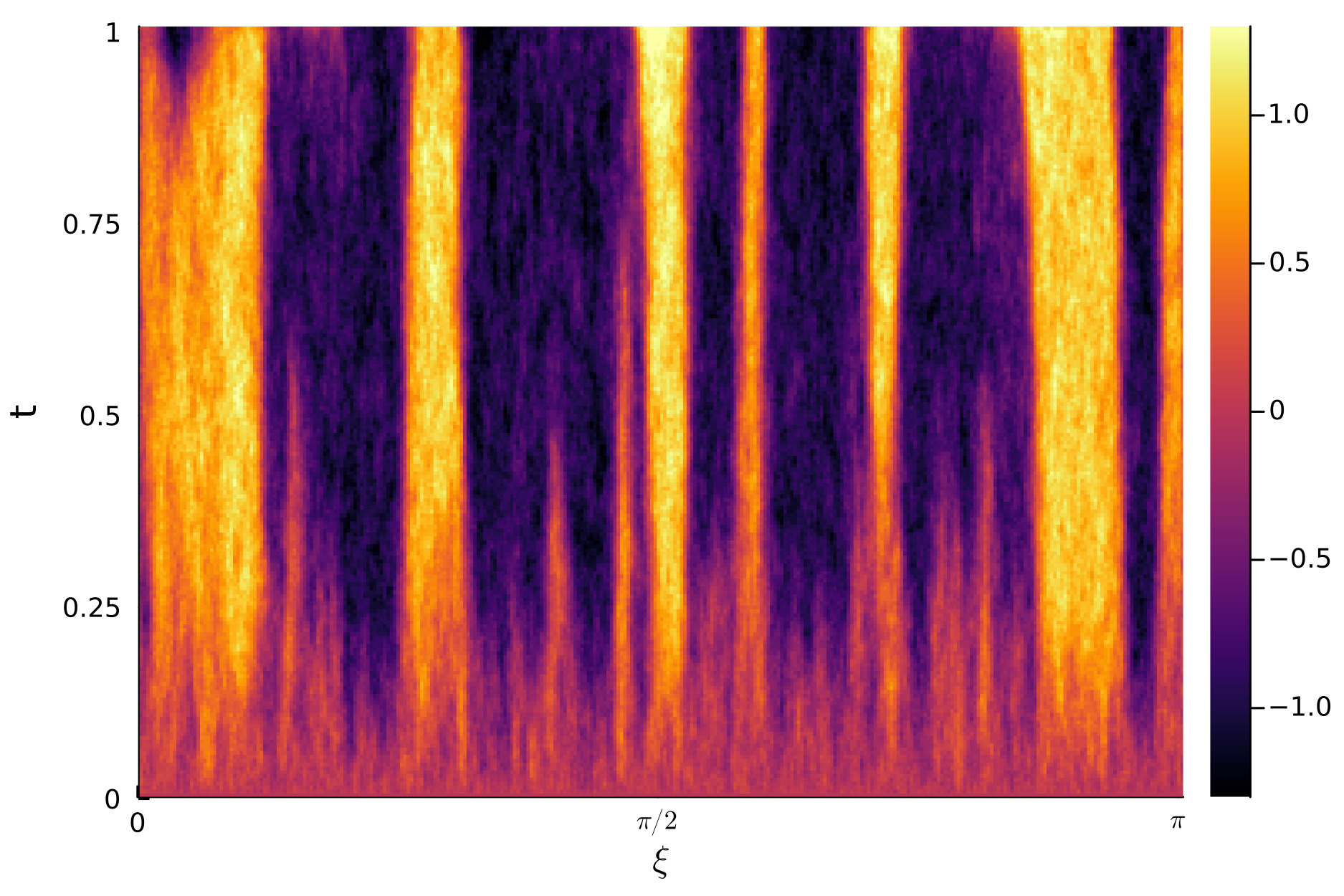} 
    \caption{Heatmap of a sample path $\Xcrc(t,\xi)$ of the guided process corresponding to \eqref{eq: AllenCahn} with conditioning state $y = P_k X_T, k = 20$.}
    \label{fig: AC_guided}
\end{figure}

\begin{figure}[ht]
    \centering
    \includegraphics[width=0.5\textwidth]{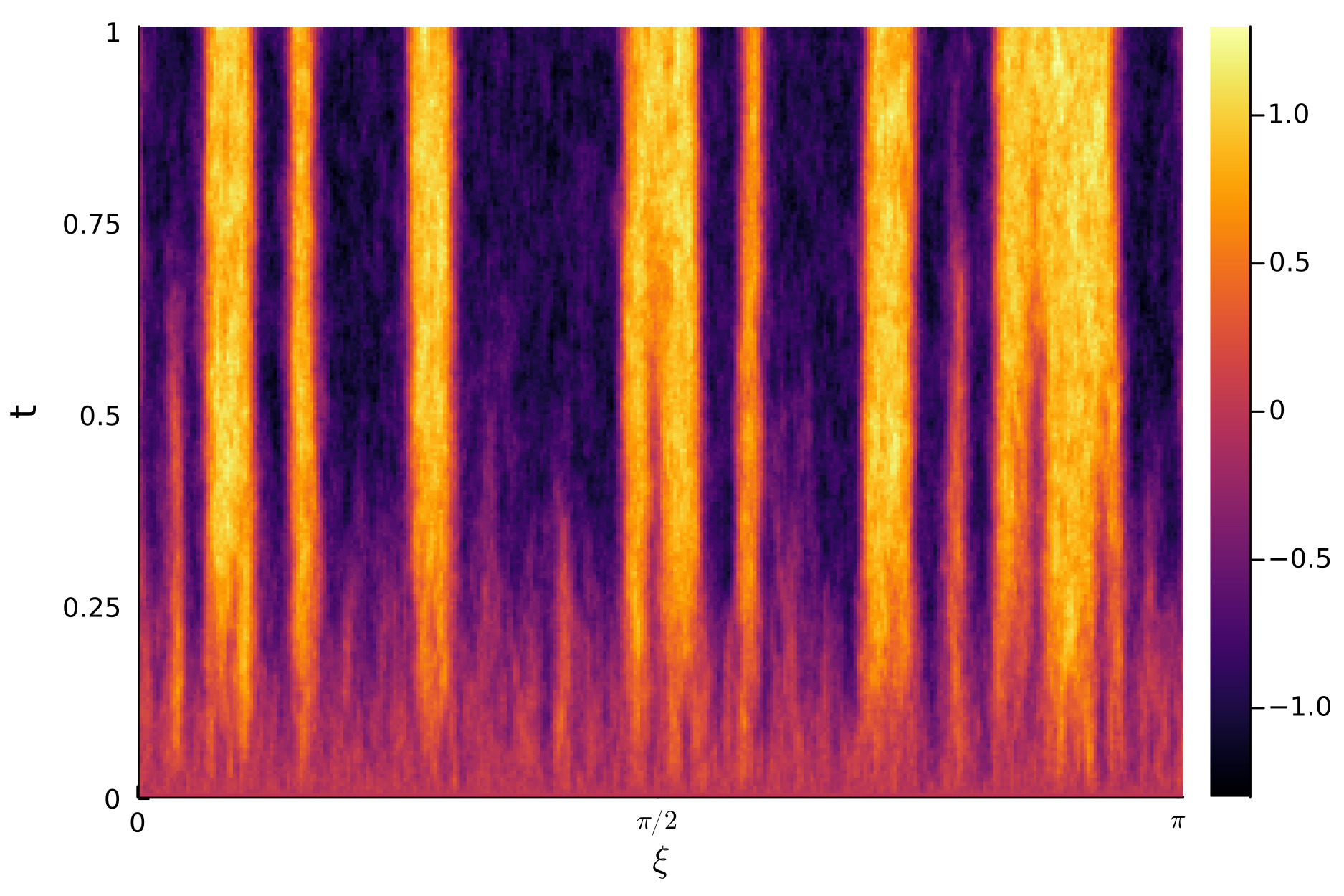} 
    \caption{Heatmap of the mean sample path of $100$ samples of the estimated diffusion bridge $\Xstr(t,\xi)$ of Equation \eqref{eq: AllenCahn} returned by Algorithm \ref{alg: MH_sampler1}.}
    \label{fig: AC_bridge}
\end{figure}

We follow the experimental setup of Example \ref{ex: MichaelisMenten}, using the observation scheme $y = P_k X_T$ with $k=20$ and the same numerical solver on an identical space-time grid.
For the parametrisation of \eqref{eq: AllenCahn} we set
\begin{align}
\label{eq: AC_params}
[\eta, \zeta, \sigma_0, \rho, \nu] = [2 \times 10^{-3}, 10, 10^7, 5 \times 10^{-6}, 1].
\end{align}
Figure \ref{fig: AC_forward} shows the heatmap of a corresponding sample path $X$. It displays the typical pattern formation of a noisy Allen-Cahn equation with steady states at $\pm 1$. The initial state $x_0$, true state $X_T$ and observed state $y$ are shown on the right in Figure \ref{fig: AC_forward}.

In Figure \ref{fig: AC_guided} we display a random sample path of the guided process $X^{\circ}$, whereas Figure \ref{fig: AC_bridge} shows the mean sample path of $100$ samples of the estimated diffusion bridge distribution after running Algorithm \ref{alg: MH_sampler1} with step size $\beta = 0.07$, $50 ~ 000$ iterations and a resulting $27 \%$ acceptance rate.

A close inspection shows that the guided process fails to capture steady states that are not represented in the data $y$, most notably around $\xi \approx \pi/8$. In contrast, the diffusion bridge mean captures all steady states displayed by the data generating path $X$, with a separation between the states that is characteristic for the Allen-Cahn equation. 

In Figure \ref{fig: AC_spectralbridges} we display the paths of the spectral modes $\{\xstr_{j_k}(t): t \in [0,T], k = 1,2 \}$, with $j_1, j_2$ as defined in \eqref{eq: j1j2}. In this case, we have $j_1 = 5$ and $j_2 = 11$.
For comparison, Figure \ref{fig: AC_spectralforward} shows the `typical' spectral sample paths of the diffusion `bridge', obtained by sampling $100 ~ 000$ paths $Y$ of Equation \eqref{eq: AllenCahn} and only keeping those paths that satisfy \eqref{eq: MMexperimentsconditioning} with $\varepsilon = 0.05$. This results in $52$ out of the $100~000$ samples being kept.
The comparison of Figures \ref{fig: AC_spectralbridges} and \ref{fig: AC_spectralforward} again shows that the 'typical' spectral path obtained through this heuristic sampling approach matches the output of Algorithm \ref{alg: MH_sampler1} quite well.

\begin{figure}[ht]
    \centering
    \includegraphics[width=0.8\textwidth]{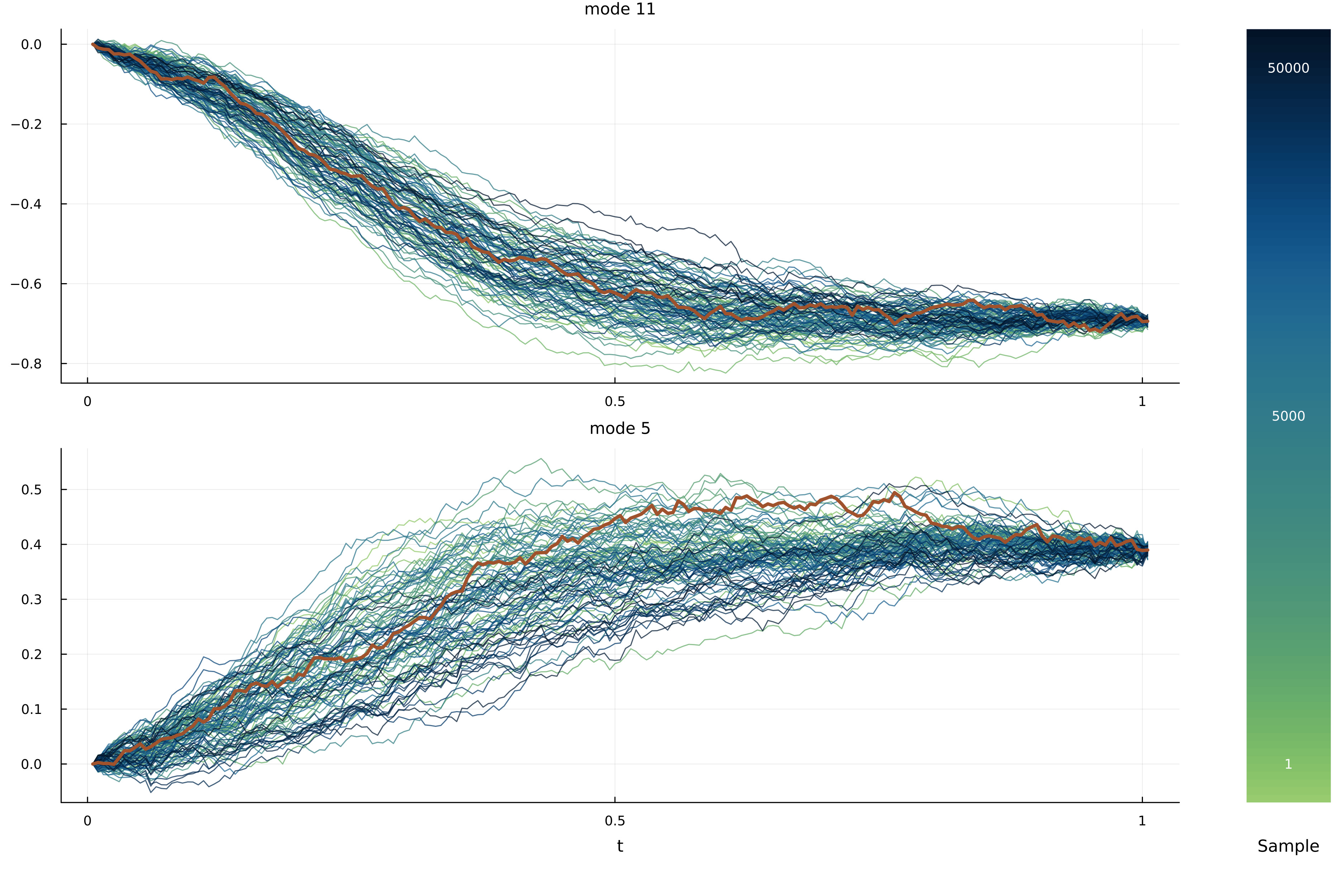} 
    \caption{Paths of two spectral modes of the diffusion bridge samples of Equation \eqref{eq: AllenCahn} returned by Algorithm \ref{alg: MH_sampler1}. Every $500$-th sample is shown. Green paths represent `earlier' samples in the Markov chain, whereas blue paths represent `later' samples. The orange path is the spectral mode of the data generating process $X$.}
    \label{fig: AC_spectralbridges}
\end{figure}

\begin{figure}[ht]
    \centering
    \includegraphics[width=0.8\textwidth]{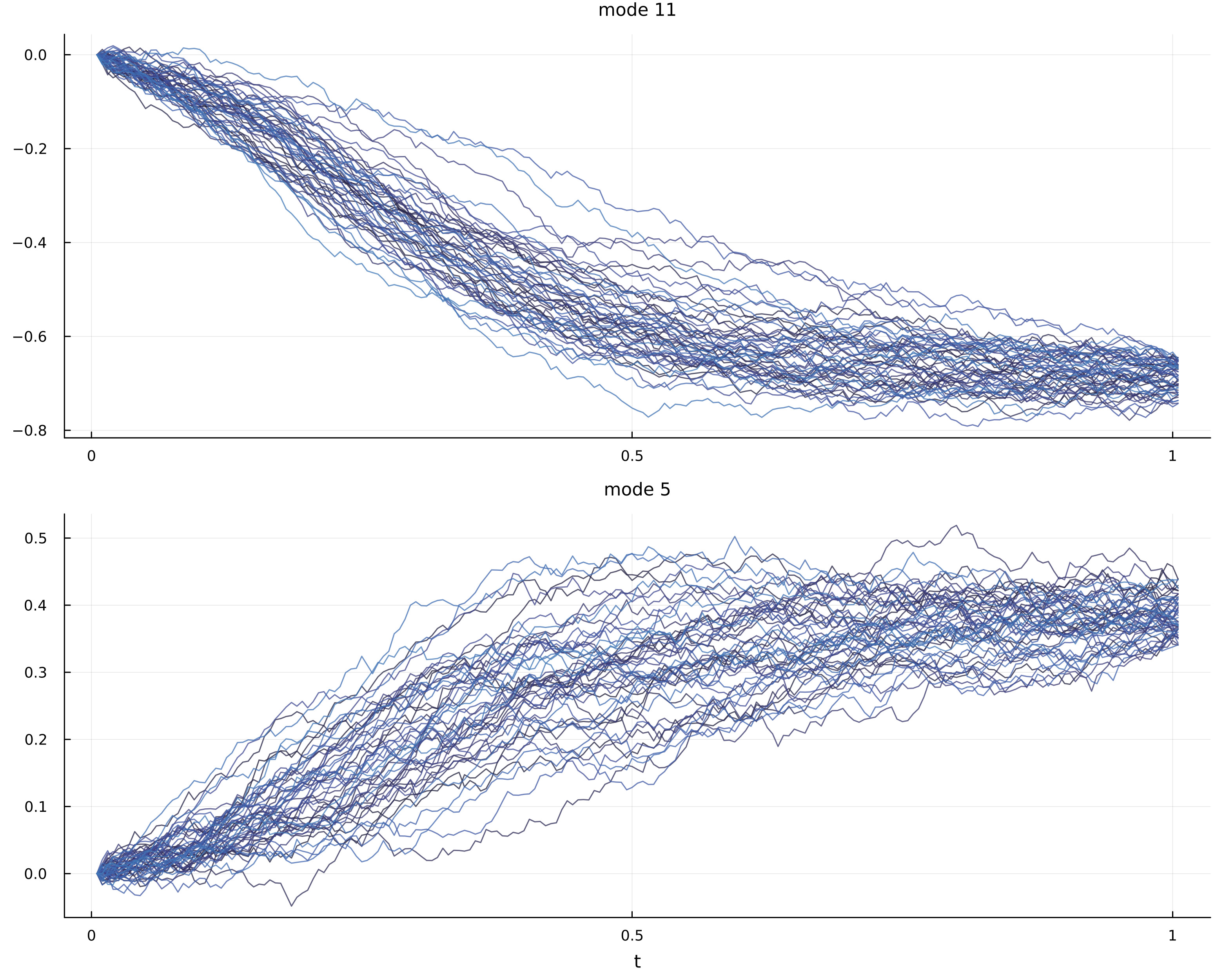} 
    \caption{Sample paths in the spectral modes $5$ and $11$ obtained by forward sampling the SPDE \eqref{eq: AllenCahn} with parameters as specified in \eqref{eq: AC_params}. Only paths that satisfy \eqref{eq: MMexperimentsconditioning} are kept.}
    \label{fig: AC_spectralforward}
\end{figure}

We remark that a precise assessment of the accuracy of the bridge sampling procedure is hard to obtain in this infinite-dimensional setting. Therefore, we provide in Appendix \ref{app: C} additional numerical validation that the result in Theorem \ref{thm: absolute_continuity} remains true when Assumptions \ref{ass: basic_assumptions_SPDE}(iii) and \ref{ass: absolute_continuity} are not necessarily met. 
\end{example}

\section{Proof of Theorem \ref{thm: Xcrc_limit_behavior}}
\label{sec: 6_proof_limit_behavior}
To ease readability of the proof of Theorem \ref{thm: Xcrc_limit_behavior}, we write $\Deltat = T-t$ for any $t < T$.
The structure of the proof follows that of the proof of Proposition 2.2 in \cite{vdMeulenBierkensSchauer2020Simulation}. 
It can be outlined as follows:
\begin{itemize}
    \item[1.] Given $Y_t= y - L_{\Deltat} \Xcrc_t$, define the Lyapunov function $V(t,Y) = \fsqrt \langle R^{-1}_{\Deltat} Y_t, Y_t \rangle.$
    \item[2.] Apply Itô's lemma on $V(t,Y)$.
    \item[3.] Bound the terms appearing in $\df V(t,Y)$.
    \item[4.] Apply a Gronwall-type inequality.
\end{itemize}

Steps 1 and 2 are covered in the following lemma.
\begin{lemma}
Let $Y_t = y - L_{\Deltat} \Xcrc_t$ and let $ V(t,Y_t) = \fsqrt \langle R^{-1}_{\Deltat} Y_t, Y_t \rangle.$
Then
\begin{align}
\begin{split}
\label{eq: dVtY}
    \df V(t,Y_t) &= -  \langle   L^*_{\Deltat} R_{\Deltat}^{-1} Y_t,   F(t,\Xcrc_t) \rangle \df t \\
    &\quad + \langle Q^{\fsqrt} L^*_{\Deltat} R_{\Deltat}^{-1} Y_t, \df \tilde{W}_t \rangle - \dfrac{1}{2} \langle Q^{\fsqrt} L_{\Deltat}^* R_{\Deltat}^{-1} Y_t, Q^{\fsqrt} L_{\Deltat}^* R_{\Deltat}^{-1} Y_t \rangle \df t \\
    &\quad + \dfrac{1}{4} \tr \left[ R_{\Deltat}^{-1} L_{\Deltat} Q L^*_{\Deltat} \right] \df t, \quad \Pcrc\text{-a.s.},
\end{split}
\end{align}
where $\tilde{W}$ is the cylindrical Wiener process $\tilde{W}_t = - \Wcrc_t$.
\end{lemma}

\begin{proof}
Since $\Xcrc$ is a mild solution to the SPDE \eqref{eq: dXcrc}, it satisfies
\begin{align*}
\Xcrc_t =  S_t x_0 + \int_0^t S_{t-s} \left[ F(s,\Xcrc_s)+ Q G(s,\Xcrc_s) \right]\, \df s  + \int_0^t S_{t-s} Q^{\fsqrt} \, \df W^{\circ}_s, \quad \Pcrc\text{-a.s.}
\end{align*}
Hence, using the semigroup property $S_{T-t} S_{t-s} = S_{T-s}$, it follows that
\begin{align*}
L_{\Deltat} \Xcrc_t &= L S_{T-t} \Xcrc_t \\
&= L \left( S_T x_0 + \int_0^t S_{T-s} \left[ F(s,\Xcrc_s)+ Q G(s,\Xcrc_s) \right]\, \df s  + \int_0^t S_{T-s} Q^{\fsqrt} \, \df \Wcrc_s\right).
\end{align*}
Setting $\tilde{W}_t := - \Wcrc_t$, it thus holds
\begin{align}
\label{eq: dYt}
\df Y_t = - L_{\Deltat} \left[F(t,\Xcrc_t) + Q G(t, \Xcrc_t) \right] \df t + L_{\Deltat} Q^{\fsqrt} \df \tilde{W}_t.
\end{align}

Next, by definition, $Q_t$ satisfies $\df Q_t = S_t Q S^*_t \df t$.
Hence, for $R_t = L Q_t L^*$, it follows that
\begin{align*}
    \df R_t = L \df Q_t L^* = L S_t Q S^*_t L^* \df t = L_t Q L_t^* \df t
\end{align*}
which gives that
\begin{align}
\begin{split}
\label{eq: dRt^-1}
\dfrac{\df}{\df t} R^{-1}_{\Deltat} &= - R^{-1}_{\Deltat} \left(\dfrac{\df}{\df t} R_{\Deltat}\right) R^{-1}_{\Deltat} = R^{-1}_{\Deltat} L_{\Deltat} Q L_{\Deltat}^* R^{-1}_{\Deltat}.
\end{split}
\end{align}
Plugging in \eqref{eq: dYt} and \eqref{eq: dRt^-1}, we get 
\begin{align}
\label{eq: d(R^-1Y)}
\begin{split}
\df \left( R_{\Deltat}^{-1} Y_t \right) &= \left( \dfrac{\df}{\df t} R^{-1}_{\Deltat}\right) Y_t \df t + R_{\Deltat}^{-1} \df Y_t \\
&= R^{-1}_{\Deltat} L_{\Deltat} Q L_{\Deltat}^* R^{-1}_{\Deltat} Y_t \df t \\
&\quad -  R_{\Deltat}^{-1} L_{\Deltat} \left[F(t,\Xcrc_t) + Q G(t, \Xcrc_t) \right] \df t + R_{\Deltat}^{-1} L_{\Deltat} Q^{\fsqrt} \df \tilde{W}_t \\
&= - R^{-1}_{\Deltat} L_{\Deltat} F(t,\Xcrc_t) \df t + R_{\Deltat}^{-1} L_{\Deltat} Q^{\fsqrt} \df \tilde{W}_t,
\end{split}
\end{align}
where in the last step we use that $Q G(t, \Xcrc_t) = Q L^*_{\Deltat} R^{-1}_{\Deltat} Y_t$.

Finally, it follows from Itô's lemma in the second line, plugging in \eqref{eq: dYt} and \eqref{eq: d(R^-1Y)} in the third line and again using $Q G(t, \Xcrc_t) = Q L^*_{\Deltat} R^{-1}_{\Deltat} Y_t$ in the last equality that 
\begin{align}
\begin{split}
\label{eq: dVtY_2}
2 \df V(t,Y_t) &= \df \langle R_{\Deltat}^{-1} Y_t, Y_t \rangle \\
&= \langle R_{\Deltat}^{-1} Y_t, \df Y_t \rangle  +  \langle \df \left( R_{\Deltat}^{-1} Y_t \right), Y_t \rangle  + \dfrac{1}{2} \tr \left[ R_{\Deltat}^{-1} L_{\Deltat} Q L^*_{\Deltat} \right] \df t \\
&= - \langle R_{\Deltat}^{-1} Y_t, L_{\Deltat} \left[F(t,\Xcrc_t) + Q G(t, \Xcrc_t) \right] \rangle \df t + \langle R_{\Deltat}^{-1} Y_t, L_{\Deltat} Q^{\fsqrt} \df \tilde{W}_t \rangle \\
&\quad -  \langle R^{-1}_{\Deltat} L_{\Deltat} F(t,\Xcrc_t), Y_t \rangle \df t  + \langle R_{\Deltat}^{-1} L_{\Deltat} Q^{\fsqrt} \df \tilde{W}_t,Y_t \rangle \\
&\quad + \dfrac{1}{2} \tr \left[ R_{\Deltat}^{-1} L_{\Deltat} Q L^*_{\Deltat} \right] \df t\\
&= - 2  \langle   L^*_{\Deltat} R_{\Deltat}^{-1} Y_t,   F(t,\Xcrc_t) \rangle \df t \\
&\quad + 2 \langle Q^{\fsqrt} L^*_{\Deltat} R_{\Deltat}^{-1} Y_t, \df \tilde{W}_t \rangle - \langle Q^{\fsqrt} L_{\Deltat}^* R_{\Deltat}^{-1} Y_t, Q^{\fsqrt} L_{\Deltat}^* R_{\Deltat}^{-1} Y_t \rangle \df t \\
&\quad + \dfrac{1}{2} \tr \left[ R_{\Deltat}^{-1} L_{\Deltat} Q L^*_{\Deltat} \right] \df t.
\end{split}
\end{align}
\end{proof}

We proceed with the more technical steps 3 and 4.
\begin{proof}{(of Theorem \ref{thm: Xcrc_limit_behavior})}
We aim to bound all terms appearing on the right hand side of \eqref{eq: dVtY}, starting with the stochastic integral term. 
For this, fix some $t_0 \in [0,T)$ and define 
\begin{align*}
    M_t = \int_{t_0}^t  \langle Q^{\fsqrt} L^*_{\Deltas} R_{\Deltas}^{-1} Y_s, \df \tilde{W}_s \rangle , \quad t \in [t_0, T).
\end{align*}
Note that $M_t$ is a square-integrable martingale with quadratic variation 
\begin{align*}
    \left[ M \right]_t = \int_{t_0}^t \langle Q^{\fsqrt} L_{\Deltas}^* R_{\Deltas}^{-1} Y_s, Q^{\fsqrt} L_{\Deltas}^* R_{\Deltas}^{-1} Y_s \rangle \df s
\end{align*}
and in particular $M_t - \frac{1}{2} \left[ M \right]_t$ is a martingale that we can bound by an exponential martingale inequality as follows;

For any sequence $(\gamma_k)_k$ of positive numbers and $t_k = T-1/k, ~k \in \N$, define the events
\begin{align*}
    E_k = \left\{ \sup_{t_0 \leq t \leq t_{k+1}} \left(M_t - \frac{1}{2} \left[ M \right]_t\right) > \gamma_k \right\}.
\end{align*}
An application of \cite{Mao2007}, Theorem 1.7.4, then shows that $\P(E_k) \leq \exp(-\gamma_k)$.

Now, let $(\gamma_k)_k$ be an arbitrary but fixed sequence that satisfies
\begin{align}
\label{eq: ass_gammak}
\sum_{k=1}^{\infty} \exp(-\gamma_k) < \infty.
\end{align}
It then follows that $\sum_{k=1}^{\infty} \P(E_k) \leq \sum_{k=1}^{\infty} \exp(-\gamma_k) < \infty$, and thus, by the Borel-Cantelli lemma, that $\P(\limsup_k E_k) = 0$.
Hence, for $\P$-a.e. $\om$ there exists some $k_0(\om)$ such that for all $k \geq k_0(\om)$
\begin{align}
    \sup_{t_0 \leq t \leq t_{k+1}} \left(M_t - \frac{1}{2} \left[ M \right]_t\right) \leq \gamma_k.
\end{align}

The remaining terms in \eqref{eq: dVtY} are easily bounded as follows.
From Assumption \ref{ass: basic_assumptions_SPDE}, it follows that there exists some $c > 0$ such that $\sup_{x \in H, t \in [0,T]} |L_{\Deltat} F(t,x)| \leq c$.
Hence, with $\|R_{\Deltat}^{-1}\| \leq \bar{c} \Deltat^{-1}$ as given in Assumption \ref{ass: Xcrc_limit_behavior}, it holds on $[0,T)$ that
\begin{align*}
    \left| \langle  L^*_{\Deltat} R_{\Deltat}^{-1} Y_t, F(t,\Xcrc_t) \rangle \right| \leq c ~ \bar{c} ~ \Deltat^{-1} |Y_t|.
\end{align*}

Moreover, by Lemma \ref{lem: trace_bound}, we have $\tr[L_t Q L_t^*] \leq \bar{c}'$ for any $t \in [0,T]$ and thus
\begin{align*}
    \left| \tr \left[ R_{\Deltat}^{-1} L_{\Deltat} Q L^*_{\Deltat} \right] \right| \leq \bar{c} \bar{c}' \Deltat^{-1}.
\end{align*}

In total, we get for any $t \in [t_0,t_{k+1}]$ that 
\begin{align}
\begin{split}
\label{eq: dVtY_3}
V(t,Y_t) &= V(t_0,Y_{t_0}) + \int_{t_0}^t \df V(s,Y_s) \\
&\leq V(t_0,Y_{t_0}) + \gamma_k + c ~ \bar{c} \int_{t_0}^{t} \Deltas^{-1} \left|Y_s \right| \, \df s \\
&\quad + \frac{\bar{c} \bar{c}'}{4} \int_{t_0}^{t} \Deltas^{-1} \, \df s.
\end{split}
\end{align}

Now define $\xi_t = \Deltat^{-1} |Y_t|^2 $. With $\|R_{\Deltat}^{-1}\| \geq \underbar{c} ~ \Deltat^{-1}$ as given in Assumption \ref{ass: Xcrc_limit_behavior} and the symmetry of $R^{-1}_{\Deltat}$, it follows that
\begin{align*}
    V(t,Y_t) = \frac{1}{2} \langle R_{\Deltat}^{-1} Y_t, Y_t \rangle \geq \fsqrt \underbar{c} ~ \Deltat^{-1} |Y_t|^2 = \fsqrt \underbar{c} ~ \xi_t.
\end{align*}
Hence, the inequality in \eqref{eq: dVtY_3} can be rewritten as 
\begin{align*}
    \xi_t &\leq M_0 + \frac{2}{\underbar{c}} \gamma_k + M_1 \int_{t_0}^t \sqrt{\xi_s} \Deltas^{-\fsqrt} \, \df s + M_2 \int_{t_0}^t \Deltas^{-1} \, \df s \\
    &=  M_0 + \frac{2}{\underbar{c}} \gamma_k + M_1 \int_{t_0}^t \sqrt{\xi_s} \Deltas^{-\fsqrt} \, \df s + M_2
    \ln \left( \frac{T-t_0}{T-t}\right),
\end{align*}
where $M_0 = \frac{2 V(t_0,Y_{t_0})}{\underbar{c}}$, $M_1 = \frac{2 c \bar{c}}{\underbar{c}}$ and $M_2 = \frac{\bar{c} \bar{c}'}{2 \underbar{c}}$.

Thus, applying the Gronwall-type inequality in Lemma \ref{lem: gronwall_inequality}, we get 
\begin{align}
\begin{split}
\label{eq: a1p2983}
    \sqrt{\xi_t} &\leq \sqrt{M_0 + \frac{2}{\underbar{c}} \gamma_k + M_2 \int_{t_0}^t \Deltas^{-1} \, \df s} + \frac{M_1}{2} \int_{t_0}^t \Deltas^{- \fsqrt} \, \df s \\
    &\quad = \sqrt{M_0 + \frac{2}{\underbar{c}} \gamma_k + M_2 \ln \left( (T-t_0)/(T-t)\right)} + M_1 \left[ \sqrt{\Delta_{t_0}} - \sqrt{\Deltat}\right].
\end{split}
\end{align}

Now, for any fixed $\varepsilon >0$, consider the sequence $\gamma_k = \ln(k^{1+\varepsilon})$. Then $(\gamma_k)_k$ satisfies \eqref{eq: ass_gammak}.
Moreover, for any $t \in [t_k,t_{k+1}]$, it holds 
\begin{align*}
    \frac{\gamma_k}{\ln\left(1/(T-t)\right)} \leq \frac{\gamma_k}{\ln\left(1/(T-t_k)\right)} = \frac{\ln(k^{1+\varepsilon})}{\ln\left(k\right)} = (1 + \varepsilon).
\end{align*}

Hence, dividing both sides in \eqref{eq: a1p2983} by $\sqrt{\ln\left(1/(T-t)\right)}$, considering $t \in [t_k, t_{k+1}]$ and taking the limit in $k$ gives that
\begin{align*}
    \limsup_{t \uparrow T} \frac{|y - L_{\Deltat} \Xcrc_t|}{\sqrt{(T-t) \ln\left( 1/(T-t)\right)}} \leq \sqrt{\frac{2}{\underbar{c}} (1+\varepsilon) + \frac{\bar{c} \bar{c}'}{2 \underbar{c}}}.
\end{align*}
The claim follows from noting that $\varepsilon >0$ was chosen arbitrarily. 
\end{proof}

The following lemma was used in the proof of Theorem \ref{thm: Xcrc_limit_behavior}.
\begin{lemma}
\label{lem: trace_bound}
Under Assumption \ref{ass: basic_assumptions_SPDE}, there exists some constant $\bar{c}'$ such that
\begin{align}
\label{eq: trace_bound}
    \tr[L_t Q L_t^*] \leq \bar{c}', \quad t \in [0,T].
\end{align}
\end{lemma}
\begin{proof}
Since $H$ is reflexive, it is well known that the family of operators $(S^*_t)_{ t \geq 0}$ is a strongly continuous semigroup on $H$. In particular, $(L^*_t)_{t \geq 0}$ is strongly continuous.
Hence, denoting by $(b_j)_j$ the standard basis in $\R^k$, the mapping 
\begin{align*}
    t \mapsto \tr[ L_t Q L_t^*] =  \sum_{j=1}^k \langle L_t Q L_t^* b_j, b_j \rangle = \sum_{j=1}^k \| Q^{\fsqrt} L_t^* b_j \|^2
\end{align*}
is continuous and thus \eqref{eq: trace_bound} follows.
\end{proof}

\section{Proof of Theorem \ref{thm: absolute_continuity}}
\label{sec: 7_proof_abs_cont}

\begin{proof}{(of Theorem \ref{thm: absolute_continuity})}

Let $(\Phi_t(\Xcrc))_{t < T}$ be the density process defined by
\begin{align*}
    \Phi_t(\Xcrc) &= \dfrac{\df \Pstr_t}{ \df \Pcrc_t}(\Xcrc) = \frac{h(t,\Xcrc_t)}{g(t,\Xcrc_t)} \frac{g(0,x_0)}{h(0,x_0)} \Psi_t(\Xcrc).
\end{align*}
Following Theorem \ref{thm: Xcrc_limit_behavior}, the random variable $\Psi_T(\Xcrc)$ is $\Pcrc$-a.s. finite and hence $\Psi_t(\Xcrc) \to \Psi_T(\Xcrc)$ $\Pcrc$-a.s.
Moreover, from Assumption \ref{ass: absolute_continuity}, it follows that $\lim_{t \to T} h(t,\Xcrc_t)/g(t,\Xcrc_t) = 1 ~\Pcrc\text{-a.s.}$ and thus, in total, 
\begin{align}
\label{eq: ap918212837}
    \lim_{t \uparrow T} \Phi_t(\Xcrc)  = \Phi_T(\Xcrc) \quad \Pcrc\text{-a.s.}
\end{align}

We will show that this convergence holds in $L^1(\Pcrc)$. Since $(\Phi_t(\Xcrc))_{t < T}$ is a martingale under $\Pcrc$, the convergence in $L^1(\Pcrc)$ then implies that 
\begin{align*}
    \Phi_t(\Xcrc) = \Ecrc[\Phi_T(\Xcrc) \mid \calF_t] \quad \Pcrc\text{-a.s.}
\end{align*}
and hence that $\Pstr$ is absolutely continuous with respect to $\Pcrc$ on $\calF_T$ with $\df \Pstr = \Phi_T(\Xcrc) \df \Pcrc$.

For this, let $r(t) := \sqrt{ (T-t) \ln(1/(T-t))}$ and define for any $m \geq 1$ the stopping time
\begin{align}
\label{eq: sigmam}
\sigma_m = T \wedge \inf\{ t \in [0,T] : \|y - L_{T-t} X_t\| \geq m ~ r(t)\}.
\end{align}

Following Equation \eqref{eq: Pstrt_Pcrc_t}, it holds for any $\calF_t$-measurable random variable $f_t$ that
\begin{align*}
    \frac{g(0,x_0)}{h(0,x_0)} \Ecrc \left[  \Psi_t(\Xcrc) ~f_t \right] = \Estr\left[\dfrac{g(t,\Xstr_t)}{h(t,\Xstr_t)}~ f_t \right].
\end{align*}
Hence, setting $f_t = \mathbbm{1}_{\{t \leq \sigma_m\}}$, we get
\begin{align}
\label{eq: a102983}
\frac{g(0,x_0)}{h(0,x_0)} \Ecrc \left[   \Psi_t(\Xcrc) \mathbbm{1}_{\{t \leq \sigma_m\}}  \right] = \Estr\left[ \dfrac{g(t,\Xstr_t)}{h(t,\Xstr_t)}\mathbbm{1}_{\{t \leq \sigma_m\}} \right].
\end{align}
We proceed by taking $\lim_{m \to \infty} \lim_{t \uparrow T}$ on both sides of \eqref{eq: a102983}. Starting with the left-hand side, it follows from the dominated convergence theorem that
\begin{align*}
    \lim_{t \uparrow T} \Ecrc \left[  \Psi_t(\Xcrc) \mathbbm{1}_{\{t \leq \sigma_m\}}  \right] &= \Ecrc \left[   \Psi_T(\Xcrc) \mathbbm{1}_{\{T \leq \sigma_m\}} \right]
\end{align*}
Here, we use that, following Lemma \ref{lem: likelihoodbound}, $\sup_t \Psi_t(\Xcrc) \mathbbm{1}_{\{t \leq \sigma_m\}} $ is bounded for any fixed $m \in \N$.
Furthermore, by the definition of $\sigma_m$ and Theorem \ref{thm: Xcrc_limit_behavior}, it holds that $\lim_m \mathbbm{1}_{\{T \leq \sigma_m\}} = \lim_m \mathbbm{1}_{\{T = \sigma_m\}} \uparrow 1$ $\Pcrc$-a.e.
Hence, by the previous display and monotone convergence, it holds that
\begin{align}
\label{eq: ap198232134}
    \lim_{m \to \infty} \lim_{t \uparrow T} \Ecrc \left[  \Psi_t(\Xcrc) \mathbbm{1}_{\{t \leq \sigma_m\}}  \right] &= \Ecrc \left[   \Psi_T(\Xcrc) \right].
\end{align}

It remains to take the limit on the right-hand side of \eqref{eq: a102983}. Applying Proposition \ref{prop: existence_diffusion_bridge} in the first line, we have
\begin{align*}
    \Estr\left[ \dfrac{g(t,\Xstr_t)}{h(t,\Xstr_t)}\mathbbm{1}_{\{t \leq \sigma_m\}} \right] &= \E\left[\dfrac{g(t,X_t)}{h(0,x_0)}\mathbbm{1}_{\{t \leq \sigma_m\}} \right] \\
    &= \E\left[\dfrac{g(t,X_t)}{h(0,x_0)}\right] - \E\left[\dfrac{g(t,X_t)}{h(0,x_0)} \mathbbm{1}_{\{t > \sigma_m\}} \right].
\end{align*}
From Lemma \ref{lem: gsigmam_vanishes}, it follows that $\lim_{m \to \infty} \lim_{t \uparrow T} \E\left[g(t,X_t)/h(0,x_0) \mathbbm{1}_{\{t > \sigma_m\}} \right] = 0$.
Moreover, following Assumption \ref{ass: absolute_continuity}, it holds that
\begin{align*}
    \lambda_2^{-1}(T-t) = \lambda_2^{-1}(T-t) ~ \E\left[\dfrac{h(t,X_t)}{h(0,x_0)}\right] \leq \E\left[\dfrac{g(t,X_t)}{h(0,x_0)}\right] \leq \lambda^{-1}_1(T-t) ~ \E\left[\dfrac{h(t,X_t)}{h(0,x_0)}\right] = \lambda^{-1}_1(T-t),
\end{align*}
which shows that $\lim_{t \uparrow T} \E\left[\dfrac{g(t,X_t)}{h(0,x_0)}\right] = 1$. We conclude that
\begin{align*}
    \lim_{m \to \infty} \lim_{t \uparrow T}\Estr\left[ \dfrac{g(t,\Xstr_t)}{h(t,\Xstr_t)}\mathbbm{1}_{\{t \leq \sigma_m\}} \right] = 1.
\end{align*}
Jointly with \eqref{eq: a102983} and \eqref{eq: ap198232134}, this shows that $\Ecrc[\Phi_T(\Xcrc)] = 1$.

To conclude the proof, note that $\E^{\circ}[\Phi_t(\Xcrc)] = 1$ for all $t < T$. Hence, by Scheffé's lemma and \eqref{eq: ap918212837}, this implies that $\lim_{t \uparrow} \Phi_t(\Xcrc) = \Phi_T(\Xcrc)$ in $L^1(\Pcrc)$.
\end{proof}

\subsection{Supplementary lemmas in the proof of Theorem \ref{thm: absolute_continuity}}
Throughout this section we work under the Assumptions of Theorem \ref{thm: absolute_continuity}.

\begin{lemma}
\label{lem: likelihoodbound}
Let $\sigma_m$ be defined as in \eqref{eq: sigmam}. Then there exists a constant $K > 0$ such that
\begin{align*}
    \Psi_t(\Xcrc) \mathbbm{1}_{\{t \leq \sigma_m\}} \leq \exp(K m) \quad \Pcrc\text{-a.s.}
\end{align*}
\end{lemma}
\begin{proof}
Per definition of $\sigma_m$, it holds on the event $\{t \leq \sigma_m\}$ that
\begin{align*}
    | y - L_{T-t} \Xcrc_t| \leq m ~ r(t), 
\end{align*}
with $r(t) = \sqrt{ (T-t) \ln(1/(T-t))}$.
Hence, plugging in $G(s,x)$ as given in \eqref{eq: G}, it follows on $\{t \leq \sigma_m\}$ for any $s \leq t$ that
\begin{align}
\label{eq: a0p2343p08}
    |G(s,\Xcrc_s)| &= |L^*_{T-s} R^{-1}_{T-s} (y - L_{T-s} \Xcrc_s)| \leq \tilde{c} ~ \bar{c} ~ m (T-s)^{-1/2} \sqrt{ \ln(1/(T-s))},
\end{align}
where we use the bound $\|R_{T-s}^{-1}\| \leq \bar{c} (T-s)^{-1}$ of Assumption \ref{ass: Xcrc_limit_behavior} and the fact that $\|L^*_{T-s}\| \leq \tilde{c}$ for some $\tilde{c} > 0$.
In total, using the boundedness of $|F(s,x)| \leq C_F$ and \eqref{eq: a0p2343p08}, we get
\begin{align*}
    \Psi_t(\Xcrc) \mathbbm{1}_{\{t \leq \sigma_m\}} &= \exp \left( \int_0^t  \Bigl \langle F(s,\Xcrc_s),  G(s,\Xcrc_s)\Bigr \rangle \, \df s \right) \mathbbm{1}_{\{t \leq \sigma_m\}} \\
    &\leq \exp \left( \int_0^t  \left| \Bigl \langle F(s,\Xcrc_s),  G(s,\Xcrc_s)\Bigr \rangle \right| \, \df s \right) \mathbbm{1}_{\{t \leq \sigma_m\}} \\
    &\leq \exp \left( \int_0^t C_F \tilde{c} ~ \bar{c} ~ m (T-s)^{-1/2} \sqrt{ \ln(1/(T-s))}  \, \df s \right) \mathbbm{1}_{\{t \leq \sigma_m\}} \quad \Pcrc\text{-a.s.}
\end{align*}
Noting that the term in the exponential is integrable over $[0,T]$, the claim follows with $K = C_F \tilde{c} ~ \bar{c} \int_0^T (T-s)^{-1/2} \sqrt{ \ln(1/(T-s))} \, \df s$.
\end{proof}

\begin{lemma}
\label{lem: gsigmam_vanishes}
Let $\sigma_m$ be defined as in \eqref{eq: sigmam}. Then
\begin{align*}
    \lim_{m \to \infty} \lim_{t \uparrow T} \E[ g(t,X_t) \mathbbm{1}_{\{t > \sigma_m\}}] = 0.
\end{align*}
\end{lemma}
\begin{proof}
Just as in the proof of \cite{vdMeulenBierkensSchauer2020Simulation}, Lemma 6.4, it can be shown that
\begin{align}
\label{eq: ap918273}
\lim_{m \to \infty} \lim_{t \uparrow T} \E \left[g(\sigma_m, X_{\sigma_m}) \mathbbm{1}_{\{t > \sigma_m\}} \right] = 0.
\end{align}

The claim then follows from the fact that 
\begin{align*}
    \E[ g(t,X_t) \mathbbm{1}_{\{t > \sigma_m\}}] &\leq C ~ \E[ h(t,X_t) \mathbbm{1}_{\{t > \sigma_m\}}] \\
    &= C \, \E[ \, \E[ h(t,X_t) \mathbbm{1}_{\{t > \sigma_m\}} \mid \calF_{\sigma_m}]] \\
    &=  C \, \E[ h(\sigma_m, X_{\sigma_m}) \mathbbm{1}_{\{t > \sigma_m\}}] \\
    &\leq C^2 \E[ g(\sigma_m, X_{\sigma_m}) \mathbbm{1}_{\{t > \sigma_m\}}].
\end{align*}
Here, we used that, following Assumption \ref{ass: absolute_continuity}, $g(t,x) \leq C_1 h(t,x)$ and $h(t,x) \leq C_2 g(t,x)$ for all $t < T, x \in H$ with $C_1 := \sup_{t \in [0,T]} \lambda_1^{-1}(t)$ and $C_2 := \sup_{t \in [0,T]} \lambda_2(t)$. 
\end{proof}

\subsection{Proof of Lemma \ref{lem: bounds_gh}}
\label{subsec: boundsghproof}

\begin{proof}
Recall that $h(t,x) = \rho_X(t,x;T,y)$ and $g(t,x) = \rho_Z(t,x;T,y)$ are the transition densities of $LX_T \mid X_t = x$ and $LZ_T \mid Z_t =x$. 
Denote by $X^{t,x}$ and $Z^{t,x}$ the processes $X$ and $Z$ initiated at $X_t = x$ and $Z_t = x$ respectively. 

We start by showing the second inequality in \eqref{eq: bounds_ghlemma}.
By the Girsanov theorem and the abstract Bayes formula we have, for any bounded and measurable functional $f$, that
\begin{align*}
    \E[f(X^{t,x}) \mid L X^{t,x}_T = y] = \frac{\rho_Z(t,x;T,y)}{\rho_X(t,x;T,y)} \E \left[f(Z^{t,x}) \frac{\df \calL(X^{t,x})}{\df \calL(Z^{t,x})}(Z^{t,x}) \mid LZ^{t,x}_T = y \right],
\end{align*}
where 
\begin{align*}
    \frac{\df \calL(X^{t,x})}{\df \calL(Z^{t,x})}(Z) = \exp \left( \int_t^T \langle \tilde{F}(s,Z^{t,x}_s), \df W_s \rangle - \frac{1}{2} \int_t^T \|\tilde{F}(s,Z^{t,x}_s) \|^2 \, \df s \right) \quad \P\text{-a.s.}
\end{align*}
is the Girsanov likelihood between $X^{t,x}$ and $Z^{t,x}$ on $C([t,T];H)$.

Denote by $(M(u))_{u \geq t}$ the continuous local martingale $M(u) = \int_t^u \langle \tilde{F}(s,Z^{t,x}_s), \df W_s \rangle$ with quadratic variation $[M]_u = \int_t^u \|\tilde{F}(s,Z^{t,x}_s) \|^2 \, \df s$.
Setting $f \equiv 1$ we get that
\begin{align}
\begin{split}
    \label{eq: ap29374}    \rho_X(t,x;T,y) &= \rho_Z(t,x;T,y)  \E \left[ \exp \left( M(T) - \frac{1}{2} [M]_T \right) \mid LZ^{t,x}_T = y \right] \\
    &\leq \rho_Z(t,x;T,y)  \E \left[ \exp \left( M(T) \right) \mid LZ^{t,x}_T = y \right]. 
\end{split}
\end{align}

From the Dambis-Dubins-Schwarz theorem it follows that there exists a real-valued Wiener process $\bar{W}$ such that
\begin{align*}
    M(u) = \bar{W}_{[M]_u} \quad \P\text{-a.s.}, u \geq t.
\end{align*}

Furthermore, by the boundedness of $\tilde{F}$, we have that
\begin{align*}
    [M]_T = \int_t^T \|\tilde{F}(s,Z^{t,x}_s) \|^2 \, \df s \leq C_{\tilde{F}}^2 (T-t).  
\end{align*} 
Hence it holds that 
\begin{align}
\begin{split}
\label{eq: a1p923874}
        \E \left[ \exp \left( M(T) \right) \mid LZ^{t,x}_T = y \right] &= \E \left[ \exp\left( \bar{W}_{[M]_T} \right) \mid LZ^{t,x}_T = y \right] \\
        &\leq \E \left[ \exp\left( \sup_{0 \leq s \leq C_{\tilde{F}}^2 (T-t)} \bar{W}_{s} \right) \mid LZ^{t,x}_T = y \right] \\
        &= \E \left[ \exp\left( \sup_{0 \leq s \leq C_{\tilde{F}}^2 (T-t)} \bar{W}_{s} \right) \right].
\end{split}
\end{align}

Noting that, for any $\xi \geq 0$, $\tilde{W}_{\xi} := \sup_{0 \leq s \leq \xi} \bar{W}_{s}$ has density $f_{\tilde{W}_{\xi}}(z) = \sqrt{2/(\pi \xi)} \exp(- z^2/(2 \xi)) \mathbbm{1}_{[0,\infty)}(z)$, it follows with $\xi_t:= C_{\tilde{F}}^2 (T-t)$ that the term on the right hand side of \eqref{eq: a1p923874} equals
\begin{align*}
    \sqrt{2/(\pi \xi_t)}  \int_0^{\infty} \exp(z) \exp(- z^2/(2\xi_t)) \, \df z &= \lim_{a \to \infty} \left[ -\exp(\xi_t/2) \text{erf}\left(\dfrac{\xi_t-z}{\sqrt{2 \xi_t}}\right)\right]_{z = 0}^{z = a} \\
    &= \exp(\xi_t/2) \left[ 1 + \text{erf}(\sqrt{\xi_t/2}) \right].
\end{align*}
Here, $\text{erf}(z) = 2/\sqrt{\pi} \int_0^z \exp(-v^2) \, \df v$ denotes the Gaussian error function.
This, jointly with $\eqref{eq: ap29374}$ and $\eqref{eq: a1p923874}$, shows the second inequality in the lemma with
\begin{align}
    \lambda(T-t) = \exp\left(\dfrac{C^2_{\tilde{F} }(T-t)}{2}\right) \left[ 1 + \text{erf} \left( \sqrt{\dfrac{C^2_{\tilde{F} }(T-t)}{2}}\right)\right].
\end{align} 

The first inequality follows by interchanging the roles of $\rho_X$ and $\rho_Z$ in \eqref{eq: ap29374}. By the same arguments, we then get
\begin{align*}
    \rho_Z(t,x;T,y) &\leq \rho_X(t,x;T,y) \lambda(T-t),
\end{align*}
which shows \eqref{eq: bounds_ghlemma}.

\end{proof}

\section{Discussion and future work}
In this article, we introduced novel methodology to sample from the infinite-dimensional diffusion bridge $\Xstr$, defined as the mild solution to a stochastic partial differential equation conditioned on a finite-dimensional, linear transformation $y = LX_T$ of the state $X_T$ for some $T >0$. 

For this, we introduced the guided process $\Xcrc$ that mimics the dynamics of the original process while being forced to the conditioning state $y$. The key aspect of our results is the absolute continuity of its law $\calLcrc$ with respect to the intractable law $\calLstr$ of the diffusion bridge. This enabled the definition of Algorithm \ref{alg: MH_sampler1}, a Metropolis-Hastings sampler targeting $\calLstr$ whose performance was illustrated in numerical examples for stochastic reaction-diffusion equations.

\paragraph{Fully observed states}
One problem that we have not addressed in this article is the case that $L$ equals the identity operator, i.e. when one `fully' observes the state $ y= X_T$ instead of a finite-dimensional transformation thereof. 
Several challenges arise in this scenario. 
Firstly, the natural candidate for the $h$-function is given by $h(t,x) = \Df_x \log p(t,x;T,y)$, where $p$ is the transition density of $X$ with respect to some suitable reference measure $\nu$ on $(H,\calB(H))$. The existence of such transition densities is not necessarily given when $H$ is infinite-dimensional. In some cases, a suitable Gaussian reference measure $\nu$ can be constructed as the invariant Ornstein-Uhlenbeck measure. However, this imposes additional assumptions on $A$ and $Q$. 

The more difficult challenge is to show that the guided process $\Xcrc$ converges to the conditioning state as $t \uparrow T$, similar to our result of Theorem \ref{thm: Xcrc_limit_behavior}. Heuristically speaking, this proves difficult to show as the convergence must occur at the same rate, simultaneously in all infinite coordinates of $\Xcrc$. 
In Theorem \ref{thm: Xcrc_limit_behavior}, it is the observation operator $L$ that maps $y$ and $\Xcrc$ to a finite dimensional subspace, therefore decreasing the degrees of freedom in the components of $L\Xcrc$, that allows us to prove convergence at a suitable rate.

While the case of observing a full state $X_T$ is certainly of mathematical interest, let us remark that the assumption of observing finite dimensional transformations thereof certainly appears to be of more practical relevance.

\paragraph{Multiple observations}
A natural extension of our work is to consider partial observations $L_i X_{t_i}$ of $X$ at multiple observation times $t_i$, possibly corrupted by some observation noise $\eta_i$. 
The task of state estimation then branches into the two well-known problems of \textit{filtering} and \textit{smoothing}.
For finite-dimensional state space models, many state-of-the-art solutions to the filtering and smoothing problems are \textit{particle-based} solutions, in which states of the latent process $X$ are estimated by a set of weighted samples. Without going into too much detail, we refer to \cite{Chopin2020Introduction} for an introduction into the rich literature in this field. 

A common challenge in particle-based inference is to construct importance sampling distributions for the unobserved latent path that are tractable, mimic the behavior of the dynamical system and are informed by the data.
We hypothesize that the path measure of the guided process derived in this paper serves as a good starting point into the construction of such distributions for partially observed stochastic PDEs. We will investigate this further in future research.

\bibliographystyle{apalike}  
\bibliography{references}

\begin{appendices}

\section{}
\label{app: A}

\begin{proposition}[Proposition \ref{prop: existence_diffusion_bridge} above.]
The mapping $h$ defined in \eqref{eq: def_h_g} satisfies Assumptions $(i)$ and $(ii)$ of Theorem \ref{thm: htransform} with $K h = 0$. Moreover, the measure $\Pstr$ defined on $\calF_T$ by
\begin{align}
\label{eq: aps98dp12973}
    \df \Pstr_t = \frac{h(t,X_t)}{h(0,x_0)} \df \P_t, \quad t < T,
\end{align}
is such that, for any bounded and measurable function $\varphi$ and $0 \leq t_1 \leq ... \leq t_n < T$, it holds
\begin{align}
\label{eq: ap129831}
    \E^{\star}[\varphi(X_{t_1},...,X_{t_n})] = \E[ \varphi(X_{t_1},...,X_{t_n}) \mid L X_T = y].
\end{align}
We call the process $X$ under $\Pstr$ the infinite-dimensional diffusion bridge (of $X$ given $LX_T = y$).
\end{proposition}
\begin{proof}
We first show that $K h = 0$. For this, it suffices to show that $h$ satisfies
\begin{align}
\label{eq: h_spacetime_harmonic}
    h(s,x) = \E[ h(t+s, X_{t+s}) \mid X_s = x]
\end{align}
for all $s,t \geq 0$ such that $s+t < T$ and $x \in H$. The claim then follows immediately from the definition of $K$. 

Recall that $h(s,x) = \rho_X(s,x;T,y)$ is defined as the density of $LX_T \mid X_s =x$, evaluated at some fixed $y \in \R^k$.
Hence, to show \eqref{eq: h_spacetime_harmonic}, let $\calA \in \calB(\R^d)$ be arbitrary and denote by $\mu_{s,t}(x,\calA)$ the Markov transition kernel of $X$. 
Then indeed for any $s,t \geq 0$ such that $t+s < T$ it holds that
\begin{align*}
    \P(L X_T \in \calA \mid X_s = x) &= \int_H \P(L X_T \in \calA \mid X_s = x, X_{s+t} = z) \, \mu_{s,s+t}(x, \df z) \\ 
    &= \int_H \P(L X_T \in \calA \mid X_{s+t} = z) \, \mu_{s,s+t}(x, \df z) \\ 
    &= \int_H \int_{\calA}  \rho_X(s+t,z;T,y) \, \df y ~\mu_{s,s+t}(x, \df z) \\
    &= \int_{\calA}  \E \left[ \rho_X(s+t,X_{s+t}; T, y) \mid X_s = x\right] \, \df y.
\end{align*}
Since $\calA$ was arbitrary, this shows \eqref{eq: h_spacetime_harmonic} for almost every $y \in \R^k$.

To show the second claim \eqref{eq: ap129831}, we write, for any $y \in \R^k$ with slight abuse of notation, $\E^y \left[ \varphi(X_t) \right]$ for the change of measure $\Pstr$ as defined in \eqref{eq: aps98dp12973} with choice $h(t,x) = \rho_X(t,x;T,y)$.
We show that $\E^{y}\left[ \varphi(X_t) \right]$ is a version of $\E\left[\varphi(X_t) \mid LX_T = y\right]$.
For this, it suffices to show that
\begin{align}
    \E \left[  \E^{LX_T} \left[ \varphi(X_t) \right] \mathbbm{1}_{\{L X_T \in \calA\}} \right] = \E \left[ \varphi(X_t) \mathbbm{1}_{\{L X_T \in \calA\}} \right]
\end{align}
for any $\calA \in \calB(\R^d)$.
Indeed we have for any $t \in [0,T)$ that
\begin{align*}
     \E \left[  \E^{LX_T}\left[ \varphi(X_t) \right] \mathbbm{1}_{\{LX_T \in \calA\}} \right] &= \int_{\calA} \E^{y}\left[ \varphi(X_t) \right] \rho_X(0,x_0;T,y) \df y \\
     &= \int_{\calA}  \E\left[ \varphi(X_t) \frac{\rho_X(t,X_t;T,y)}{\rho_X(0,x_0;T,y)}\right] \rho_X(0,x_0;T,y) \df y \\
    &= \E \left[ \varphi(X_t) \int_{\calA}   \rho_X(t,X_t;T,y)  \df y \right] \\
    &= \E \left[ \varphi(X_t) \E\left[ \mathbbm{1}_{\{LX_T \in \calA\}} \mid X_t \right] \right] \\
    &= \E \left[ \varphi(X_t) \mathbbm{1}_{\{LX_T \in \calA\}} \right].
\end{align*}
Here we use the definition of $\rho_X(0,x_0;T,y)$ in the first line, the definition of $\E^{y}\left[ \varphi(X_t) \right]$ in the second line, Fubini's theorem in the third step and the law of total expectation in the last step. 
The claim now follows from a standard cylindrical argument, see e.g. \cite{EthierKurtz2009markov}, Proposition 4.1.6.   
\end{proof}

\begin{proposition}[Proposition \ref{prop: existence_guided_process} above.]
The mapping $g$ defined in \eqref{eq: def_h_g} satisfies Assumptions $(i)$ to $(iii)$ of Theorem \ref{thm: htransform} with $K g(t,x) = \langle F(t,x), \Df_x g(t,x) \rangle.$
In particular, there exists a unique measure $\Pcrc$ on $\calF_T$, defined by 
\begin{align*}
    \df \Pcrc_t &= \dfrac{g(t,X_t)}{g(0,x_0)} \exp\left( - \int_0^t \langle F(s,X_s), \Df_x \log g(s, X_s) \rangle \, \df s \right) \df \P_t, \quad t < T,
\end{align*} 
such that $X$ under $\Pcrc$ is the unique mild solution to the SPDE
\begin{align}
\label{eq: dXcrc2}
        \df \Xcrc_t = \left[ A \Xcrc_t + F(t,\Xcrc_t) + Q \Df_x \log g(t, \Xcrc_t) \right] \df t + Q^{\fsqrt} \df \Wcrc_t, \quad t \in [0,T),
\end{align} 
where $(\Wcrc_t)_{t \in [0,T)}$ is a $\Pcrc$-cylindrical Wiener process. 
\end{proposition}
\begin{proof}
We begin by showing that $g$ satisfies Assumption (iii) of Theorem \ref{thm: htransform}.
By definition of $g$ we have 
\begin{align*}
    g(t,x) &= \dfrac{1}{\sqrt{(2 \pi)^k \det(R_{T-t})}} \exp \left( - \dfrac{1}{2} | R_{T-t}^{-\fsqrt} (y - L_{T-t} x) |^2 \right),
\end{align*}
from which it follows that $g$ is continuous and bounded on $[0,S]$ for any $S < T$.
Moreover, $g$ is Fréchet differentiable in $x$ with bounded derivative 
\begin{align*}
    \Df_x g(t,x) = g(t,x) \left( L^*_{T-t} R^{-1}_{T-t} (y - L_{T-t} x)\right).
\end{align*}
This proves Assumption (iii). To show the first assumption, denote by $\tilde{K}$ the infinitesimal generator of the Ornstein-Uhlenbeck process, defined as in \eqref{eq: def_inf_gen} by substituting $T$ with the transition semigroup of $Z$.
By the Fréchet differentiability of $g$, it follows from \cite{manca2009fokker}, Theorem 4.1, that
\begin{align*}
    (K g)(t,x) = \tilde{K} g(t,x) + \bigl \langle F(t,x), \Df_x g(t,x) \bigr \rangle.
\end{align*}
Reiterating the arguments in the proof of Proposition \ref{prop: existence_diffusion_bridge} gives that $\tilde{K} g = 0$. Hence, it follows that 
\begin{align*}
    (g^{-1} Kg)(t,x) = \bigl \langle F(t,x), \Df_x \log g(t,x) \bigr \rangle,
\end{align*}
which shows Assumption (i) and the existence of the continuous local martingale $(E^g_t)_{t < T}$. Lastly, to show that $E^g$ is a martingale, it suffices to note that 
\begin{align*}
    \Df_x \log g(t,x) = L^*_{T-t} R^{-1}_{T-t} (y - L_{T-t} x)
\end{align*}
is Lipschitz in $x$, uniformly in $t \in [0,S]$ for any $S < T$. The martingale property then follows from $\cite{Piepersethmacher24Classexponentialchangesmeasure}$, Lemma 3.6.
\end{proof}

\section{}
\label{app: B}

\begin{lemma}
\label{lem: gronwall_inequality}
Let $t \mapsto \zeta(t)$ be nonnegative and continuously differentiable on $[t_0,t_1)$ and let $t \mapsto f(t) $ be nonnegative and continuous on $[t_0, t_1).$ Assume that $t \mapsto u(t)$ is a nonnegative and continuous function on $[t_0, t_1)$ such that
\begin{align*}
    u(t) \leq \zeta(t) + \int_{t_0}^{t_1} f(s) \sqrt{u(s)} \df s. 
\end{align*}
Then $u(t)$ satisfies
\begin{align*}
    u(t) \leq \left( \sqrt{\zeta(t_0) + \int_{t_0}^{t} |\zeta'(s)| \df s} + \frac{1}{2} \int_{t_0}^{t} f(s) \df s \right)^2, \quad t \in [t_0,t_1).
\end{align*}
\end{lemma}
\begin{proof}
    This follows from \cite{Agarwal2005Generalization}, Theorem 2.1. In their notation we have $n=1, w_1(x) = \sqrt{x}, W_1(x) = 2 \sqrt{x}$.
\end{proof}

\section{}
\label{app: C}

We provide additional numerical evidence that the results in Theorem \ref{thm: absolute_continuity} hold true in the setting of Example \ref{ex: AllenCahn}, even when Assumption \ref{ass: basic_assumptions_SPDE}(iv) and the condition of Lemma \ref{lem: bounds_gh} are not met. 
Assume that Equation \eqref{eq: Phi_T} holds true. It then follows that 
\begin{align}
\label{eq: ao123pu13o8}
    \rho_X(0,x_0;T,y) = \rho_Z(0, x_0;T,y) \Ecrc[ \Psi_T(X^{\circ,y})].
\end{align}
Here, to underline the dependence on $y$, we add a superscript $y$ to the guided process $\Xcrc$ with end state $L\Xcrc_T = y$.
Furthermore, we denote the corresponding guiding term $\Gcrc$ of $X^{\circ,y}$ by $G^y$.

Fix $x_0, T$ and denote by $\pi(y) = \rho_X(0,x_0;T,y)$ the density of $L X_T \mid X_0 = x_0$. 
It then follows from \eqref{eq: ao123pu13o8} that, for samples $\mathbb{X}^{\circ,y} =  \{X_i^{\circ,y}, i = 1,...,n\}$, the estimator 
\begin{align}
\label{eq: pi(y)}
    \hat{\pi}(y,\mathbb{X}^{\circ,y}) =  \rho_Z(0, x_0;T,y) \dfrac{1}{n} \sum_{i=1}^n \Psi_T(X^{\circ,y}_i)
\end{align}
is an unbiased estimator of the otherwise intractable density $\pi(y)$. This allows us to numerically sample from the distribution of $LX_T$ in two ways:
\begin{itemize}
    \item[1.] By approximating the solution to the Allen-Cahn equation \eqref{eq: AllenCahn} by a numerical SPDE solver and evaluating $LX_T$.
    \item[2.] By sampling from $\calL(LX_T)$ via an MCMC scheme based on the unbiased estimator $\hat{\pi}(y,\mathbb{X}^{\circ})$.
\end{itemize}   

For the numerical solver we use a spectral Galerkin approximation and a semi-implicit Euler-Maruyama scheme to approximate the resulting SODE. 
For the second option we implement the correlated pseudomarginal (CPM) sampler as introduced in \cite{Deligiannidis18/CorrelatedPM}. The proposals in $y$ are simple random walk proposals. For the proposals in the latent variable $\Xcrc$ we use the same pCN scheme as introduced in Algorithm \ref{alg: MH_sampler1}.
The complete algorithm is outlined in Algorithm \ref{alg: CPM_sampler}.

\begin{algorithm}[H]
\label{alg: CPM_sampler}
\caption{CPM Sampler of $L X_T$}
\LinesNotNumbered  
    \KwIn{SPDE Parameters $A,~F, ~Q$ and $x_0$, observation operator $L$, gridded domain $D \times [0,T]$, iterations $N$, step size $\beta$, step size $\rho \in [0,1)$}
    \KwOut{Samples $(y_i)_{i=0}^N$ of $L X_T$.}
    \textbf{Initialize:} Draw $y_0 \sim \calN(0, \beta^2 I)$, draw a Wiener process $W$ and compute $X^{\circ,y_0} = \solve(A,F+Q G^{y_0},Q,W)$\;
    \For{$i = 0... ~N-1$}{
        \textbf{Proposal} \\
                \quad(i) ~~Draw $v \sim \calN(0, I)$ and set $y^{\circ} =  y_i + \beta v$ \; 
                \quad(ii) ~~Draw a Wiener process $V$ and set $W^{\circ}= \sqrt{1 - \rho^2} W + \rho V$\;
            \quad(iii) ~Compute $X^{\circ,y^{\circ}} = \solve(A,F+QG^{y^{\circ}},Q,W^{\circ})$\;
        \textbf{Update}  \\
            \quad(i) 
        ~Compute $M = \min\left(1, \dfrac{\hat{\pi}(y^{\circ},X^{\circ,y^{\circ}})}{\hat{\pi}(y_i,X^{\circ,y_i})} \right)$\;
            \quad(ii) Draw $U \sim \text{Unif}(0,1)$\; 
        \quad \quad \If{$U < M $}{
        \quad \quad Set $y_{i+1} = y^{\circ}$, $W = W^{\circ}$ and $X^{\circ,y} = X^{\circ,y^{\circ}}$\;
    }
    \quad \quad \Else{
        \quad \quad Set $y_{i+1} = y_i$.
    }
        }
\end{algorithm}

We employ Algorithm \ref{alg: CPM_sampler} for the Allen-Cahn equation \eqref{eq: AllenCahn} with parametrisation
\begin{align}
\label{eq: a238412}
    [\eta, \zeta, \sigma_0, \rho, \nu] = [1 \times 10^{-2}, 10, 10^7, 3 \times 10^{-6}, 1]
\end{align}
and initial value $x_0(\xi) = 0.5 \sin(4\xi), \xi \in [0, \pi].$ As observation operator we take $L = P_4$, in which case the algorithm returns samples of the first four spectral modes of the Allen-Cahn equation. 
The chain is run with step size $\beta = 0.03$ and $\rho = 0.1$ for $N = 100~000$ iterations of which we discard the first $50~000$ as burn-in. The resulting acceptance percentage equals $24\%$.

To compare the chain output with the empirical distribution of $LX_T$, we numerically solve the SPDE \eqref{eq: AllenCahn} and evaluate $LX_T$ for a total of $10~000$ samples. 
Figure \ref{fig: AC_QQ} shows a quantile-quantile plot of the two sample sets for the first and fourth spectral modes of $X_T$. It indicates that Algorithm \ref{alg: CPM_sampler} correctly samples from the target distribution of $L X_T$, providing further evidence that Equation \eqref{eq: pi(y)} and hence the results of Theorem \ref{thm: absolute_continuity} remain valid, even if Assumption \ref{ass: basic_assumptions_SPDE}(iv), and possibly Assumption \ref{ass: absolute_continuity}, are not satisfied.

\begin{figure}[ht]
    \centering
    \includegraphics[width=0.9\linewidth]{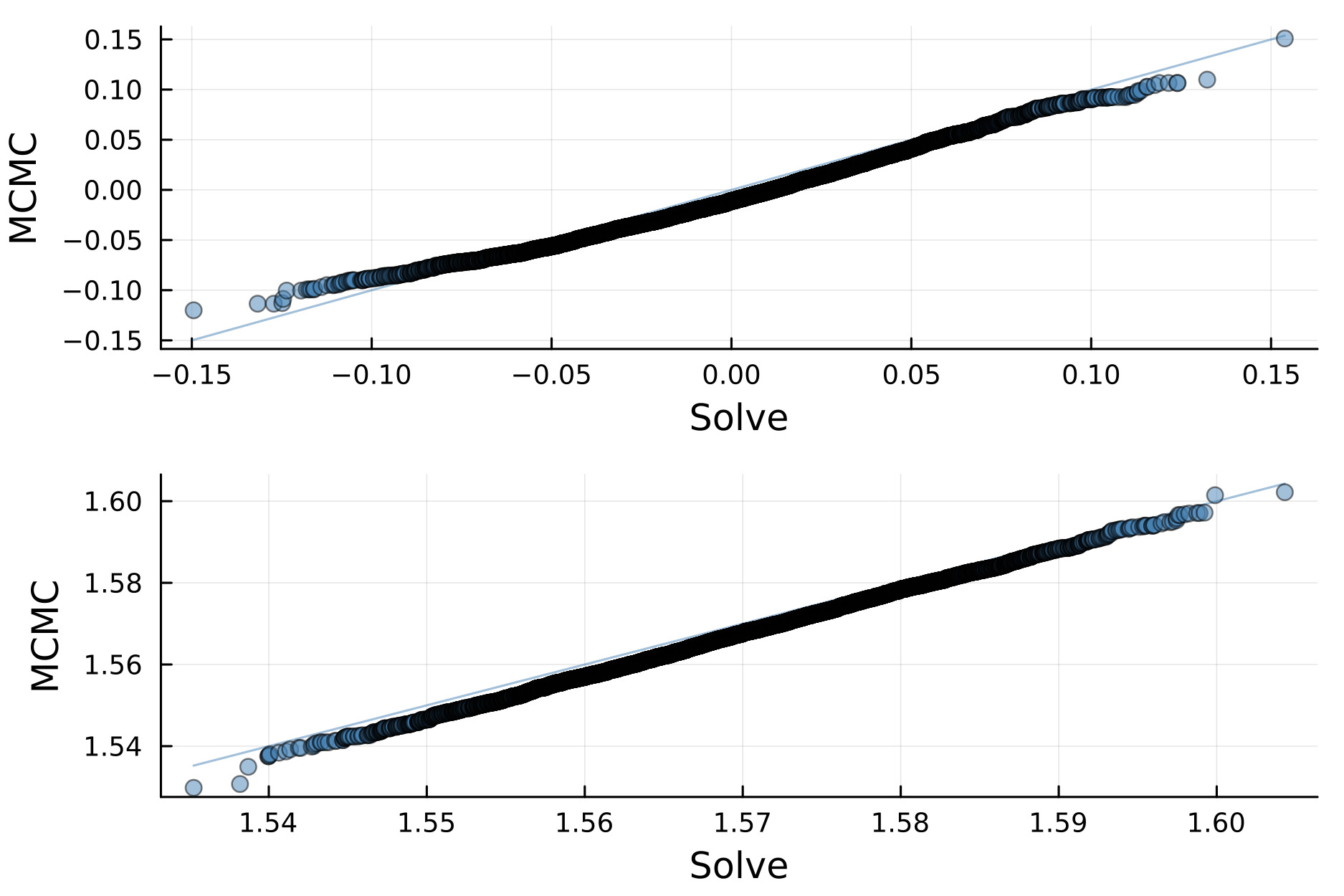}
    \caption{Quantile-quantile plot of samples of the first and fourth spectral modes of $LX_T$. Solve: samples obtained by approximating the solution to the Allen-Cahn equation \eqref{eq: AllenCahn} by a numerical SPDE solver and evaluating $LX_T$. MCMC: samples obtained by Algorithm 2, with the first $50~000$ out of $100~000$ samples discarded as burn-in. Top: first spectral mode. Bottom: fourth spectral mode.}
    \label{fig: AC_QQ}
\end{figure}

\end{appendices}

\end{document}